\documentclass[11pt,a4paper]{article}
\usepackage{amssymb,amsmath,amsopn,amsthm,graphicx,makeidx}
\usepackage[all,2cell,arc]{xy} \UseAllTwocells
\usepackage{latexsym}
\usepackage{color}

\renewenvironment{proof}{\noindent {\bf{Proof.}}}{\hspace*{3mm}{$\Box$}{\vspace{9pt}}}

\newcommand*{\CB}{\mathop{{\rm CB}}\nolimits}

\newcommand{\KG}{\mathop{{\rm KG}}\nolimits}

\newcommand*{\ov}{\overline}

\newcommand*{\bsm}{\left(\begin{smallmatrix}}
\newcommand*{\esm}{\end{smallmatrix}\right)}
\newcommand*{\bp}{\begin{pmatrix}}
\newcommand*{\ep}{\end{pmatrix}}

\newcommand*{\fty}{\infty}

\newcommand*{\wh}{\widehat}
\newcommand*{\wt}{\widetilde}

\newcommand*{\xr}{\xrightarrow}

\newcommand*{\mB}{\mathcal{B}}

\newcommand*{\al}{\alpha}
\newcommand*{\be}{\beta}
\newcommand*{\del}{\delta}
\newcommand*{\eps}{\varepsilon}
\newcommand*{\ga}{\gamma}

\newcommand*{\lam}{\lambda}
\newcommand*{\Lam}{\Lambda}
\newcommand*{\om}{\omega}
\renewcommand*{\phi}{\varphi}

\newtheorem{theorem}{Theorem}[section]
\newtheorem{cor}[theorem]{Corollary}
\newtheorem{lemma}[theorem]{Lemma}
\newtheorem{prop}[theorem]{Proposition}

\theoremstyle{remark}
\newtheorem{example}[theorem]{Example}
\theoremstyle{remark}

\theoremstyle{definition}

\theoremstyle{remark}

\renewcommand{\marginpar}[1]{ }

\newcommand{\Hom}{\operatorname{Hom}}

\begin{document}

\title{Krull--Gabriel dimension of domestic string algebras}
\author{Rosanna Laking\footnote{rosanna.laking@manchester.ac.uk, School of Mathematics, Alan Turing Building, University of Manchester,
Manchester M13 9PL, UK}, Mike Prest\footnote{mprest@manchester.ac.uk, School of Mathematics, Alan Turing Building, University
of Manchester, Manchester M13 9PL, UK} and Gena Puninski\footnote{punins@mail.ru, Department of Mechanics and Mathematics,
Belarusian State University, Praspekt Nezalezhnosti 4, Minsk 220030, Belarus}}

\maketitle

\abstract{We calculate, confirming a conjecture of Schr\"{o}er, the Krull--Gabriel dimension of the category of modules over any
domestic string algebra, as well as the Cantor--Bendixson rank of each point of its Ziegler spectrum.  We also determine the
topology on this space.}

\tableofcontents

\thanks{This paper was started during a visit of the third author to the University of Manchester, supported by EPSRC grant EP/K022490/1, and was completed during a visit of the second author to the Belarusian State University.  We thank both universities and EPSRC for their support.}


\section{Introduction}\label{secintro}\marginpar{secintro}

The Krull--Gabriel dimension, ${\rm KG}(R)$, of a ring was introduced as a variant of the Gabriel dimension for the category of
finitely presented modules by Geigle in his thesis, see \cite{Gei85} (the terminology ``Krull--Gabriel dimension" was introduced
later).  This dimension also turns out \cite[Lemma B.9]{KraHab} to be the m-dimension of the lattice of pp
formulas for $R$, equivalently of the lattice of pointed finite-dimensional modules; that dimension is a modification of the
elementary Krull dimension, introduced by Garavaglia \cite{Gardim}, and developed further in \cite[\S 8]{Zie} and \cite[Chpt.~10]{PreBk}.

A result of Auslander - see \cite[Prop. 7.2.8]{PreNBK}, says that $R$ has finite representation type if and only if $\KG(R)= 0$.
It is also known, see a remark in \cite[p.~135]{Gei86}, or use factorizable systems of morphisms, see \cite[0.2]{PreMor}, that
this dimension is undefined (or ``$=\infty$") for any wild algebra. But also many tame algebras have undefined KG-dimension.
For instance, see \cite[Cor.~4.2]{Gei86} or \cite[7.5]{HarPre},  this is the case for tubular algebras, and, see \cite[Prop. 2]{Sch00a},
for non-domestic string algebras.

It has been conjectured, see \cite[7.2.17]{PreBk}, that a finite-dimensional algebra $R$ has Krull--Gabriel dimension if and only if
it is (tame) domestic. There is some, admittedly rather limited, evidence to support this conjecture. A strong form of the conjecture
(see \cite[9.1.15]{PreBk}) proposes that, if $R$ is a finite-dimensional algebra and if $\KG(R)$ exists, then $\KG(R)$ is finite.

Herzog \cite{HerzKG1} and Krause \cite{Kra98} proved that no finite-dimensional algebra has KG-dimension 1.  Geigle \cite{Gei85}
proved that if $R$ is tame and hereditary then $\KG(R)=2$ (and also established this value for certain other classes of algebras,
\cite{Gei86}).   A few examples of algebras with larger Krull--Gabriel dimension were known. For instance, Burke and Prest \cite{BurPre}
and Schr\"{o}er \cite{Sch00a} considered series of string algebras with KG-dimensions taking every finite value $\geq 3$.  In a
recent paper, \cite{Pun12}, the third author showed that any 1-domestic (non-degenerate) string algebra has KG-dimension $3$.

It was conjectured by Schr\"oer (see \cite[8.1.22]{PreBk}) that for any domestic string algebra $R$ its Krull--Gabriel dimension
equals $n+2$, where $n$ is the maximal length of a path in the bridge quiver of $R$. In this paper we will prove this conjecture.
This completes the calculation of Krull--Gabriel dimension for string algebras, because every non-domestic string algebra has
KG-dimension $\infty$.  It follows that the Ziegler spectrum of a domestic string algebra is a $T_0$ space.

This is a companion paper to \cite{P-P14}, where the classification of indecomposable pure-injective modules was completed. After
the heavy combinatorics of \cite{P-P14}, we can reap the benefits of having that classification to hand and the arguments in this
paper are more direct and guided by our intuition.

We will make use of the Ziegler spectrum - a topology on the set of indecomposable pure-injectives - in particular, its
Cantor--Bendixson analysis.  Because there are no superdecomposable pure-injectives, the Cantor--Bendixson rank of a point is equal
to its KG-dimension and a mixture of topological and algebraic approximation arguments allows us to give a recipe, in terms of the
bridge quiver of $R$, for the rank of each string module.  Neighbourhood bases of open sets for the 1- and 2-sided points already
can be extracted from \cite{P-P04} and \cite{P-P14}; we extend this to find neighbourhood bases for Pr\"ufer points.  In these
computations we will use the lattice, ${\rm pp}_R$, of pp formulas, calculating the m-dimensions of certain intervals and hence
the KG-dimensions of points.  Having computed the dimensions of the Pr\"ufer points, we will use elementary duality to obtain the
ranks of the adic points.  By that point we have enough information to compute the ranks of the generic points.  Indeed, we give
a complete description of the topology, in that we describe a neighbourhood basis of open sets for each point.

\vspace{4pt}

Throughout $K$ is an algebraically closed field and the rings $R$ considered will be $K$-algebras.  We expect that none of the results
depend on $K$ being algebraically closed.  The category of left $R$-modules is denoted $ R\mbox{-}{\rm Mod} $; $R\mbox{-}{\rm mod}$
denotes the category of finitely-presented modules.
We will use the notation ${\rm pp}_R$ for the lattice of pp formulas (in one free variable) for $R$-modules, equivalently for the lattice of pointed
finitely presented modules.

In this paper we have to assume a fair bit of background, referring to other sources for many definitions, basic facts and results.
We do, however, define, or at least describe, the main ideas using a number of illustrative examples in order to convey the ideas.
Sources on material specific to string algebras include the original paper \cite{B-R}, the very readable introductory sections in
\cite{Sch97}), and introductory sections of various papers, such as \cite{P-P14}, of the second and third authors. For
pure-injectivity, the Ziegler spectrum and various other concepts and techniques arising from the model theory of modules, we use
\cite{PreNBK} as a comprehensive source but there are many other accounts.

Most of the finite- and infinite-dimensional modules that we are interested in are described in terms of strings and bands.  First
we introduce the combinatorics of strings and bands as well as the bridge quiver which shows how the bands are connected.  Then we
describe the associated modules and we recall the combinatorial description of the morphisms between them.  We also describe the
Ziegler spectrum - the topological space which has the indecomposable pure-injectives for its points.  In Section \ref{secrkdim}
we introduce the ranks and dimensions, the computation of which is the main aim of this paper.  That is accomplished in Section
\ref{secrecp} and we give a `recipe' for computing, not just the global rank, but also the rank of every point of the Ziegler
spectrum.  It is easy to apply:  first one computes the bridge quiver and from that it is a simple matter to read off the ranks
of points.

For each type of point, having determined how to compute the rank, we find a neighbourhood basis.  Thus we obtain a complete description
of the Ziegler spectrum, extending the results of \cite{BurPre}, on modules over the algebras $\Lambda_n$, to all domestic string
algebras.

\section{The algebras}\label{secalgs}\marginpar{secalgs}

Throughout, $R$ will be a {\bf string algebra} (and usually domestic).  So $R=KQ/I$ is the path algebra of a quiver with relations
where:

\noindent $\bullet$ every vertex of $Q$ has at most two incoming arrows and at most two outgoing arrows;

\noindent $\bullet$ for every arrow $\alpha$ there is at most one arrow $\beta$ with $\beta\alpha \neq 0$ and at most one arrow
$\gamma$ with $\alpha \gamma \neq 0$, and

\noindent $\bullet$ $I$ can be generated by monomials.

The Kronecker algebra $\xymatrix{*+={\circ}\ar@/^0.4pc/[r]^{\al}\ar@/_0.4pc/[r]_{\be} & *+={\circ}}$ is an example, as is the quiver
below with relations $\ga\be= 0$ and $\del\ga= 0$.

$$
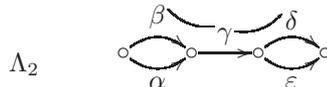
\Lam_2 \hspace{1cm}
\vcenter{
\def\labelstyle{\displaystyle}
\xymatrix@R=20pt@C=20pt{%
*+={\circ}\ar@/_0.5pc/[r]_{\al} \ar@/^0.5pc/[r]^{\be}="be" & *+={\circ}\ar[r]^{\ga}="ga" &
*+={\circ}\ar@/_0.5pc/[r]_{\eps}\ar@/^0.5pc/[r]^{\del}="del" & *+={\circ}\\
\ar@/_0.2pc/@{-}"ga"+<4pt,2pt>;"del"+<-4pt,2pt>
\ar@/_0.2pc/@{-}"be"+<4pt,3pt>;"ga"+<-4pt,4pt>
}
}
$$\label{Lam2}

More generally $\Lambda_n$ is obtained by linking $n$ Kronecker quivers in a similar way.  Here is $\Lambda_3$.

$$
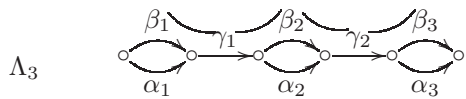
\Lam_3 \hspace{1cm}
\vcenter{
\def\labelstyle{\displaystyle}
\xymatrix@R=20pt@C=20pt{%
*+={\circ}\ar@/_0.5pc/[r]_{\al_1} \ar@/^0.5pc/[r]^{\be_1}="be1" & *+={\circ}\ar[r]^{\ga_1}="ga1" &
*+={\circ}\ar@/_0.5pc/[r]_{\al_2}\ar@/^0.5pc/[r]^{\be_2}="be2" & *+={\circ}\ar[r]^{\ga_2}="ga2" &
*+={\circ}\ar@/_0.5pc/[r]_{\al_3}\ar@/^0.5pc/[r]^{\be_3}="be3" & *+={\circ}\\
\ar@/_0.2pc/@{-}"be1"+<4pt,3pt>;"ga1"+<-4pt,3pt>
\ar@/_0.2pc/@{-}"ga1"+<4pt,3pt>;"be2"+<-4pt,4pt>
\ar@/_0.2pc/@{-}"be2"+<4pt,3pt>;"ga2"+<-4pt,4pt>
\ar@/_0.2pc/@{-}"ga2"+<4pt,3pt>;"be3"+<-4pt,4pt>
}
}
$$\label{Lam3}

The next example has the quiver shown with relations $\be^2= \ga^2= 0$ and the longer relation $\be\al\ga= 0$ (which is shown by a
solid curve).

$$
X_1 \hspace{1cm}
\vcenter{%
\def\labelstyle{\displaystyle}
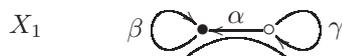
\xymatrix@R=24pt@C=20pt{%
*+={\bullet}\ar@{}[]*{}="1"\ar@(dl,ul)^{\beta}[]&
*+={\circ}\ar@{}[]*{}="2"\ar_{\alpha}[l]
\ar@(ur,dr)^{\gamma}
\ar@/^0.6pc/@{-} "1"+<-6pt,-10pt>;"2"+<6pt,-10pt>
}
}
$$\label{X1}

A {\bf string} is a walk along arrows and inverse arrows in $Q$ (avoiding zero relations and immediate backtracks) and a {\bf band}
is a walk which returns to its starting point and is {\bf primitive} - not a proper power of any string - and (a technically useful
restriction) has the form $\alpha c \beta^{-1}$ for some arrows $\alpha$, $\beta$ and finite string $c$ (we allow the empty string).
We will consider finite strings and also (singly and doubly) infinite strings, describing a string of the form $u_1 u_2\dots$, where
each $u_i$ is a direct or inverse arrow and $u_1 \dots u_n$ is a string for every $n$, as a {\bf 1-sided (infinite) string}, and one
of the form $\dots u_{-1} u_0 u_1 \dots$ as a {\bf 2-sided string}.

An example of a string, $\alpha \gamma \alpha^{-1} \beta^{-1} \alpha$ over the algebra $X_1$ above, is shown by the following diagram
(our convention is to draw direct arrows from the upper right to the lower left, and inverse arrows from the upper left to the lower
right).

$$
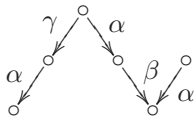
\vcenter{%
\def\labelstyle{\displaystyle}
\xymatrix@R=12pt@C=7pt{%
&&*+={\circ}\ar[dl]_{\ga}\ar[dr]^{\al}&&\\
&*+={\circ}\ar[dl]_{\al}&&*+={\circ}\ar[dr]^{\be}&&*+={\circ}\ar[dl]^{\al}\\
*+={\circ}&&&&*+={\circ}
}
}
$$

Taking a result for a definition, we will say that a string algebra is {\bf domestic} if there are just finitely many bands (see
\cite[\S 11]{RinAC1}).  We say that two bands $b$ and $b'$ are {\bf cyclically equivalent} if $b'$ is a cyclic permutation of $b$,
and {\bf equivalent} if $b'$ is cyclically equivalent to either $b$ or $b^{-1}$ (note that only one of these may occur).  We will
often use the term ``band" for the equivalence class of a band and refer to a particular member of that class as a representative.
If $n$ is the number of equivalence classes of bands over $R$, then $R$ is said to be $n$-{\bf domestic}. For instance the algebra
$X_1$ above is 1-domestic, with just the band $b= \al\ga\al^{-1}\be^{-1}$ up to equivalence, whereas $\Lambda_2$ has two bands (up
to equivalence) $\al\be^{-1}$ and $\eps\del^{-1}$.

Let $G_{2,3}$ denote the Gelfand--Ponomarev algebra whose quiver consists of one vertex and two loops $\al, \be$, with relations
$\al\be= \be\al= \al^2= \be^3= 0$. This algebra is not domestic - it is easy to see infinitely many primitive cyclic walks - or note
that we have different bands, $\al\be^{-1}$ and $\al\be^{-2}$, beginning with the same letter (cf.~\ref{domfact} below).

\vspace{4pt}

In this paper we will deal only with domestic string algebras, equivalently for every arrow $\al$, there is at most one band
beginning with $\al$.  This and other required facts about the combinatorics of domestic string algebras can be found in
\cite{RinAC1, Sch97} and \cite{Pun07}; we state the following from \cite[\S 11]{RinAC1} for ease of reference.  Band modules are
defined in Section \ref{secband}.

\begin{theorem}\label{domfact}\marginpar{domfact}  Suppose that $R$ is a domestic string algebra and that $b$, $b'$ are bands
(hence each begins with a direct letter and ends with an inverse letter).

\noindent (1)  If $b$ and $b'$ have the same first letter then $b=b'$.

\noindent (2) If $M$ is a $b$-band module, $M'$ is a $b'$-band module and they have an isomorphic simple submodule then $b$ is
equivalent to $b'$.  In particular if both $b$ and $b'$ contain the string $\beta^{-1}\alpha$ then $b'$ is a cyclic permutation of $b$.
\end{theorem}

A {\bf socle pair} of a band $b = \al \dots \be^{-1}$ is the string $\be^{-1}\al$ or a substring of $b$ of the form $\del^{-1}\ga$.
Such a pair defines, where the arrows meet, an element in the socle of the corresponding band modules (Section \ref{secband}) and
in any string module $C(w)$ (Section \ref{secstring}) when $w$ contains the pair.  Dually, a substring of $b$ of the form
$\del\ga^{-1}$ is a {\bf top pair}, since it defines a simple submodule of the corresponding band module modulo its radical.
So, by \ref{domfact}(2), over a domestic string algebra different bands have no socle pair, nor (by a dual argument) top pair, in common.

\section{Strings, bands and bridges}\label{seccomb}\marginpar{seccomb}

\subsection{String orderings and m-dimension}\label{S-string}

At each vertex $s$ of a quiver $Q$ we partition the direct and inverse arrows entering $s$ into two sets $H_{s,\pm 1}$ such that, when a walk arrives at $s$,
it enters by an direct or inverse arrow from one set and leaves {\it via} a letter in the other set.  More precisely we require that each set contain at most one direct arrow and at most one inverse arrow, and that if
$l_1, l_2\in H_{s,i}$ are {\em letters} (that is, arrows or inverse arrows) then $l_1^{-1}l_2$ is not a string.  The strings of length greater than
one are placed into $H_{s,\pm 1}$ according to their first letters.  We also put into $H_{s,i}$ a string $1_{s,i}$ of length
$0$ so as to be able to refer to walks which terminate at $s$.

For example, with reference to the diagram below where the relations are $\delta\alpha=0 =\gamma \beta$, the letters
pointing to $s$ are $\alpha, \beta, \gamma^{-1}, \delta^{-1}$.  If we put $\alpha$ into $H_{s,1}$ then it follows that
$\ga^{-1} \in H_{s,-1}$, also $\beta$ must be put into $H_{s,-1}$, and then it follows $\del^{-1} \in H_{s,1}$.  In the diagram
we show $H_{s,1}$ to the right and $H_{s,-1}^{-1}$ to the left of $s$.

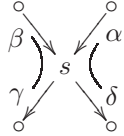
$$
\vcenter{%
\def\labelstyle{\displaystyle}
\xymatrix@R=14pt@C=9pt{%
*+={\circ}\ar[dr]_(.3){\be}="be"&& *+={\circ}\ar[dl]^(.3){\al}="al"\\
&{s}\ar[dl]_(.7){\ga}="ga"\ar[dr]^(.7){\del}="del"&\\
*+={\circ}&&*+={\circ}
\ar@/_0.3pc/@{-}"al"+<-6pt,-3pt>;"del"+<-4pt,3pt>
\ar@/^0.3pc/@{-}"be"+<6pt,1pt>;"ga"+<6pt,3pt> }}
$$

We totally order the strings in each of these sets by setting $c<d$, for $c,d\in H_{s,i}$, if: $d=c\al d'$ for some $d'$ and
$\al$; or if $c=d\be^{-1}c'$ for some $c'$ and $\be$; or if $c=e\ga^{-1}c'$ and $d=e\del d'$ for some $c',d',\ga, \del$.  When it
comes to looking at an occurrence of a particular vertex $s$ in the diagrammatic representation of a string, one of the sets will
be representing choices for that string ``to the right" while the other set represents choices ``to the left".  The meaning of
this is purely local (to that occurrence of that vertex) but the terminology is useful.  All this extends naturally to 1-sided
infinite strings, each of these being fitted in, simply by comparing any particular finite string to a long-enough initial segment
of the given infinite string; we use $\wh H_{s,i}$ for these larger totally ordered sets.

Take, for instance, the Kronecker algebra $\wt A_1$, and let $s$ be the vertex where $\al$ ends. Then one could take $H_{s,1}$ to
consist of all strings beginning (on the left) with $\al$, together with $1_{s,1}$; in which case $H_{s,-1}$ consists of all
string beginning with $\be$, and $1_{s,-1}$.  Each set $H_i$ is ordered as a chain; for example $1_1< \al$ in $H_1$ and
$\be\al^{-1}< \be\al^{-1}\be$ in $H_{-1}$.  In fact $H_1$ is the following chain of type $\om + \om^*$, where $\om^*$ denotes
the ordering opposite to that of $\omega$:

$$
1_1< \al\be^{-1}< (\al\be^{-1})^2< \dots \qquad \ldots < (\al\be^{-1})^2\al < \al\be^{-1}\al< \al\,.
$$

To place the infinite word $u= (\al\be^{-1})^{\fty}$ in relation to these, note that the strings of the form $(\al\be^{-1})^n\al$
are greater than $u$ and $u>(\al\be^{-1})^m$ for each $m$.

The m-dimension of a modular lattice $L$ is defined by setting it to be $0$ if $L$ has finite length and, beyond that, by
collapsing, at each step of a possibly transfinite induction, intervals of finite length; the reader may wish to look now at the
precise definition in Section \ref{secmdim}. For instance the m-dimension of the chain $\om + \om^*$ equals $1$:  at the first
step of the m-dimension analysis we collapse together all the points in the copy of $\om$, and all the points in the copy of
$\om^*$, but nothing else, obtaining a 2-element lattice, which then collapses at the second step.

\vspace{4pt}

For a more complicated example, consider the following 1-domestic string algebra:

$$
X_3 \hspace{1cm}
\vcenter{
\def\labelstyle{\displaystyle}
\xymatrix@R=20pt@C=20pt{%
*+={\circ}\ar@(ul,dl)_{\alpha}\ar@(ur,dr)^{\beta}
}
}
$$\label{X3}

\vspace{2mm}

\noindent with relations $\al^2= \be^2= \al\be=0$. Let us include in $H_1$ (for the unique vertex $s$) all strings
beginning with $\be$ or $\be^{-1}$, and $1_1$; so $H_{-1}$ consists of all strings that start with $\al$ or $\al^{-1}$,
and $1_{-1}$. Note that $H_1$ contains the following descending chain of type $\om^*$:

$$
\ldots < \be(\al^{-1}\be)^3  < \be(\al^{-1}\be)^2 < \be\al^{-1}\be < \be.
$$
Each interval in this poset is a chain of type $\om+ \om^*$:  for instance the interval between $\be\al^{-1}\be$ and $\be$
consists of an ascending chain $\be\al^{-1}\be (\al\be^{-1})^n$, $n=1, \dots$ and, above it, a descending chain
$\be\al^{-1}\be (\al\be^{-1})^m\al$, $m= 1, \dots$.  From this it is easily checked that the m-dimension of $H_1$ equals $2$.

\vspace{4pt}

We define the {\bf m-dimension} of a string $u\in \wh H_i$ as the infimum of m-dimensions of intervals $[c, d]$ in $H_i$
such that $c< u < d$ in $\wh H_i$.  For instance from the above description it follows that the string $u= (\al\be^{-1})^{\fty}$
over $\wt A_1$ has m-dimension $1$.  The consonance between this definition and that of the Krull--Gabriel dimension of string
modules is \ref{kgeqmdim} and \ref{mdimtp}.

For another example, consider the 1-sided (periodic) string $v= (\be\al^{-1})^{\fty}\in \wh H_1$ over $X_3$. Note that $v$ is the
infimum of the descending chain of finite strings appearing above, with consecutive intervals of this chain having m-dimension $1$.
It is easily derived that the m-dimension of $v$ is $2$.

The calculation of the m-dimension of the chains $H_i$ is straightforward (see \cite[Thm. 4.3]{Sch00b}).

\section{The bridge quiver}\label{S-bridge}

Recall that $R$ is a domestic string algebra.

It is easy to see that a band $b$ cannot overlap $b^{-1}$ in any string and $b$ cannot touch a different band $b'$. Also two copies of
$b$ cannot overlap each other in a string, for otherwise $b$ would contain a substring of the form $\al\dots\al$ which contradicts
\cite[3.4]{Pun07}.  Furthermore we have the following.

\begin{lemma}\label{overlap}
Let $R$ be a domestic string algebra, $b= \al\dots \be^{-1}$ be a band and $c= d\dots e$ (where $d, e$ are direct or inverse arrows) be
a cyclic permutation of another band. Then $b$ cannot overlap $c$ in any string.
\end{lemma}
\begin{proof}
Looking at socles and tops we see that $b$ cannot be a subword of $c$ or vice versa. Thus it suffices to consider the following
string

$$
u= \al x d y \be^{-1} z e\,,
$$
where $b=\al x dy \be^{-1}$, $c= d y \be^{-1}z e$, and $x,y,z$ denote strings of unspecified length. We claim that the
following is a string:

$$
v= \al x dy\be^{-1}z edy\be^{-1}\,.
$$

Otherwise $v$ contains a forbidden path $t$ as a subword whose letters are either all direct or all inverse. Note that the following
underbraced and overbraced parts of $v$ are strings: the first equals $u$, and the second is an initial part of $c^2$.

$$
\underbrace{\al x dy\be^{-1}z e}dy\be^{-1} \qquad \mbox{and} \qquad \al x
\overbrace{dy \be^{-1}z edy\be^{-1}}
$$

Thus $t$ must protrude to the left and right of their common (braced twice) substring $dy \be^{-1}z e$. But then
$c= dy \be^{-1}z e$ consists entirely of inverse letters, a contradiction.

Thus $v$ is a string. Because $R$ is domestic, it follows that $v= b^n$ for some $n$, so $c$ is a substring of $b^n$. Then $b$ and $c$
have a common socle or top, hence $c$ is equivalent to $b$, contradiction.
\end{proof}

The bridge quiver of $R$ shows the bands and the strings which directly connect them. Suppose that $b$ and $b'$ are different
bands and $u$ is a string such that $bub'$ is again a string. Following \cite[p.~664]{Sch00b} say that $u$ is a \textbf{bridge} from
$b$ to $b'$ if $u$ contains
no band as a substring and neither can $u$ can be written as $u= u_1 u_2$ with both $u_1, u_2$ nonempty and such that $u_1 b'' u_2 $
is a string for some band $b''$.

Then the {\bf bridge quiver}, $\mB=\mB (R)$, has for its vertices a chosen representative for each {\em cyclic} equivalence class of
bands and an arrow, labelled $c$, from $b$ to $b'$ if $c$ is a bridge from $b$ to $b'$. Since $R$ is domestic, it is easily checked
that $\mB$ is finite and has no oriented cycles (cf. \cite[\S 4.4]{Sch97}).

For instance, over $X_1$, let us choose $b= \al\ga\al^{-1}\be^{-1}$ and $b^{-1}$ as the vertices of $\mB$. Then $\mB$ is the following
quiver:

$$
\vcenter{%
\xymatrix@C=16pt@R=12pt{%
*+{b}\ar@/^0.4pc/[rr]^{\al\ga\al^{-1}}*+={b}\ar@/_0.4pc/[rr]_{\al\ga^{-1}\al^{-1}}&&*+={\, b^{-1}}
}}
$$

\vspace{2mm}

\noindent the bridge $b\xr{\al\ga\al^{-1}} b^{-1}$ being visible in the following string:

$$
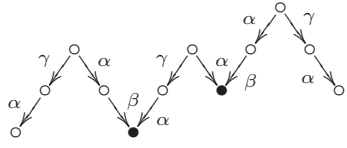
\vcenter{
\def\labelstyle{\displaystyle}
\xymatrix@R=9pt@C=5pt{%
&&&&&&&&&*+={\circ}\ar[dl]_{_{\al}}\ar[dr]^{_{\ga}}&&\\
&&*+={\circ}\ar[dl]_{_{\ga}}\ar[dr]^{_{\al}}&&&&*+={\circ}\ar[dl]_{_{\ga}}\ar[dr]^{_{\al}}&&
*+={\circ}\ar[dl]^{_{\be}}&&*+={\circ}\ar[dr]_{_{\al}}&\\
&*+={\circ}\ar[dl]_{_{\al}}&&*+={\circ}\ar[dr]^{_{\be}}&&*+={\circ}\ar[dl]^{_{\al}}&&*+={\bullet}&&&&*+={\circ}\\
*+={\circ}&&&&*+={\bullet}&&&&&&&
}
}
$$

\vspace{2mm}

\noindent where the marked points show where the bands join the bridge.

For another example, if we choose $b= \al\be^{-1}$, $b'= \eps\del^{-1}$ and their inverses for the vertices of the bridge quiver over $\Lam_2$, then the bridge
quiver is the disjoint union of $b'\xr{\eps\ga} b$ and the inverse of this, $b^{-1}\xr{\ga^{-1}\eps^{-1}} b'^{-1}$.

To give one more example, consider the following 3-domestic string algebra

$$
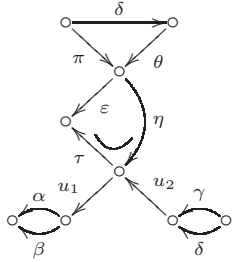
\vcenter{
\xymatrix@R=12pt@C=14pt{%
&*+={\circ}\ar[dr]_{\pi}\ar[rr]^{\del}&&*+={\circ}\ar[dl]^{\theta}\\
&&*+={\circ}\ar[dl]^{\eps}\ar@/^0.8pc/[dd]^{\eta}="eta"&&\\
&*+={\circ}&&&\\
&&*+={\circ}\ar[ul]^{\tau}="tau"\ar[dl]_(.6){u_1}&&\\
*+={\circ}&*+={\circ}\ar@/_0.4pc/[l]_{\al}\ar@/^0.4pc/[l]^{\be}&&*+={\circ}\ar[ul]_{u_2}&
*+={\circ}\ar@/_0.4pc/[l]_{\ga}\ar@/^0.4pc/[l]^{\del}
\ar@/^0.4pc/@{-}"eta"+<-10pt,-6pt>;"tau"+<6pt,9pt>
}
}
$$

\vspace{2mm}

\noindent where the relations are $\eta\pi=0$, $\eps\theta=0$, $\tau\eta=0$, $u_1\eta=0$, $\tau u_2=0$, $\be u_1=0$ and $u_2\del= 0$.  We choose as representatives of bands up to cyclic equivalence $b_1= \al \be^{-1}$, $b_2= \pi\del^{-1}\theta^{-1}$, $b_3= \ga\del^{-1}$ and their inverses.  Then the bridge quiver is the disjoint union of the diagram below and its ``inverse".

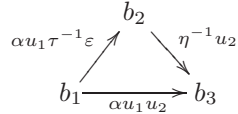
$$
\vcenter{%
\xymatrix@R=16pt@C=12pt{%
&*+{b_2}\ar[dr]^{\eta^{-1}u_2}&\\
*+={b_1}\ar[ur]^{\al u_1\tau^{-1}\eps}\ar[rr]_{\al u_1 u_2}&&*+{b_3}\\
}}
$$

It may be that choices of representatives for bands can affect the shape of the bridge quiver, though we know no example, but the questions we investigate here are not sensitive
to this. From now one we will assume that we have fixed representatives for cyclic equivalence classes of  bands, hence vertices of
the bridge quiver $\mB$.

If $w$ is a string then (by Lemma \ref{overlap}) it can be uniquely written as $w= c_1 b_1^{k_1} c_2\dots c_s b_s^{k_s} c_{s+1}$,
where $b_i\in \mB$, $k_i\geq 0$, possibly $\infty$ when $i=1, s$ and where each $c_i$ is a (possibly empty) string which contains no
band. Though the $c_i$ ($i\neq 1, s+1$) may not be bridges, we may consider
$c_i$ as a string connecting $b_{i-1}$ to $b_{i}$, hence the corresponding path $b_1 c_2\dots c_s b_s$.
We say that the {\bf band-length} of $w$ is $s-1$. Note that if $b'_i$ is another representative of the cyclic class of $b_i$, then
$b_i'$ is a substring of $b_i^2$, therefore (by enlarging $k_i$ if necessary) we can get a string $w'$ containing at least $s$ bands
from some other choice,
$\mB'$, of representative bands.  So, for instance, the maximal possible band-length of a string does not depend on the choice of $\mB$.

If $b$ is a band then the \textbf{right indent} of $b$ is the maximal band-length of a string starting with $b$, and the {\bf left indent}
is the maximum possible band-length of a string starting with $b^{-1}$. For instance over $X_1$, the right indent of the band
$b= \al\ga\al^{-1}\be^{-1}$ equals $1$, and its left indent is $0$.

\section{The modules}\label{secpts}\marginpar{secpts}

Since we compose actions of arrows from right to left it is natural to consider left $R$-modules.

\subsection{Pure-injective modules and the Ziegler spectrum}\label{secpi}\marginpar{secpi}

Over an artin algebra, $R$, finite-length modules are pure-injective.  More generally, a module over such a ring is {\bf pure-injective}
if it is a direct summand of a direct product of finite-dimensional modules (over general rings we cannot use this as the definition).
The indecomposable pure-injective modules over domestic string algebras were classified in \cite{P-P14}; we will briefly introduce
these modules; more details can be found in that paper.

Over any ring there is just a set of indecomposable pure-injectives up to isomorphism and they were organised into a topological space
by Ziegler.  This space, the (left) {\bf Ziegler spectrum}, $_R{\rm Zg}$, is (quasi)compact and has a basis of compact open sets
(see \cite[5.1.22]{PreNBK}).  That basis can be described in various ways; here are two.  Given any morphism $A\xrightarrow{f}B$
in $R\mbox{-}{\rm mod}$ we set $(f)=\{ N\in \, _R{\rm Zg}: (A,N)/{\rm im}(f,N) \neq 0\}$ - the set of indecomposable pure-injectives
$N$ such that there is a morphism from $A$ to $N$ which does not factor through $f$.  Also, given any pair $\phi>\psi$ of pp formulas,
set $(\phi/\psi)= \{ N\in\,  _R{\rm Zg}: \phi(N)/\psi(N)\neq 0\}$ - the set of points on which the pp-pair $\phi/\psi$ is open
(this is Ziegler's original definition).  The sets of the form $(\phi/\psi)$ are exactly those of the form $(f)$ and form a basis
of open sets for the topology on $_R{\rm Zg}$.

\begin{theorem}\label{isoldens}\marginpar{isoldens} (see \cite[5.3.36, 5.3.37]{PreNBK}) If $R$ is an artin algebra then the isolated,
that is open, points of $_R{\rm Zg}$ are exactly the finite-dimensional indecomposables and, together, these are dense in $_R{\rm Zg}$.
\end{theorem}

\subsection{String modules; finite- and infinite-dimensional}\label{secstring}\marginpar{secstring}

To every string $u$, finite or infinite, we assign a {\bf string module} $M(u)$.  Roughly, this is obtained, from the corresponding
walk through the quiver, by placing a 1-dimensional vector space at each vertex on the walk (including the starting vertex if there
is one), then taking the direct sum of these vector spaces and equipping the result with an $R$-module structure, with the actions
of the arrows given by $u$.  We give an example, and a precise definition can be found at \cite{B-R}, \cite{RinAC1}.  The module
$M(u)$ is indecomposable and it is also isomorphic to $M(u^{-1})$.

$$
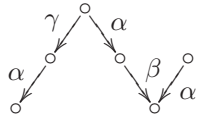
\vcenter{%
\def\labelstyle{\displaystyle}
\xymatrix@R=12pt@C=7pt{%
&&*+={\circ}\ar[dl]_{\ga}\ar[dr]^{\al}&&\\
&*+={\circ}\ar[dl]_{\al}&&*+={\circ}\ar[dr]^{\be}&&*+={\circ}\ar[dl]^{\al}\\
*+={\circ}&&&&*+={\circ}
}
}
$$

\noindent This represents the string module $M(u)$ where $u= \al\ga\al^{-1}\be^{-1}\al$, a basis for which is shown by small circles,
with the action of the path algebra on these basis elements shown by the arrows (absence of an arrow acting on a basis element should
be interpreted as the zero action).

Over a domestic string algebra (\cite[\S 11, Prop.~1]{RinAC1}), any infinite string $w$ which is not completely
periodic\footnote{If $b$ is a band then the completely periodic string $^\infty b^\infty$ is not associated to any indecomposable
pure-injective.  Indeed it is the image, under a representation embedding as in the following section, of $_{K[T]}K[T,T^{-1}]$,
the pure-injective hull of which is far from indecomposable.} is {\bf almost periodic}, meaning eventually periodic in each direction
but, if 2-sided, not completely periodic.  Given such a string and looking to, say, the right, there is, by eventual periodicity in
that direction, a map, the {\bf shift endomorphism} of $M(w)$ which either moves basis elements on the right further to the right
or moves them inwards (eventually annihilating them).  In the first case we say that $w$ is {\bf expanding} to the right, otherwise
{\bf contracting} (\cite[\S 3]{RinAC1}).  Correspondingly the associated string module $C(w)$ (see Section \ref{seclist}) will be formed using the direct
product (if expanding), or direct sum (if contracting), of the spaces generated by the basis elements on the right.  Similarly in
the other direction; so a 2-sided string could be expanding in both directions, contracting in both directions or mixed.

For example, the string $\be (\al\be^{-1})^{\fty}$ over $X_3$ is expanding and so $C(w)$ is the direct product module and has its
shift endomorphism annihilating the left-most basis element and moving every other one two places to the right.

$$
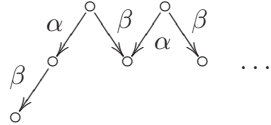
\vcenter{
\def\labelstyle{\displaystyle}
\xymatrix@R=14pt@C=8pt{%
&&*+={\circ}\ar[dl]_{\al}\ar[dr]^{\be}&&*+={\circ}\ar[dl]^{\al}\ar[dr]^{\be}&&\\
&*+={\circ}\ar[dl]_{\be}&&*+={\circ}&&*+={\circ}&\dots\\
*+={\circ}&&&&&&
}}
$$

\vspace{2mm}

For another example, the two-sided string $w= {}^{\fty} (\be\al^{-1})\be (\al\be^{-1})^{\fty}$ over $X_3$ is expanding on the right
and contracting on the left. Thus the corresponding mixed module $C(w)$, is the submodule of the direct product module which consists
of all sequences which are eventually $0$ to the left.

$$
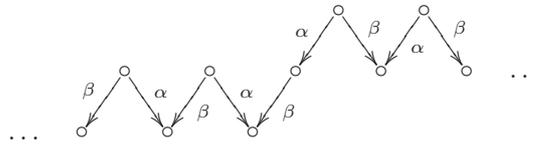
\vcenter{
\xymatrix@R=16pt@C=10pt{%
&&&&&&&*+={\circ}\ar[dl]_{\al}\ar[dr]^{\be}&&*+={\circ}\ar[dl]^{\al}\ar[dr]^{\be}&&\\
&&*+={\circ}\ar[dl]_{\be}\ar[dr]^{\al}&&*+={\circ}\ar[dl]^{\be}\ar[dr]^{\al}&&*+={\circ}\ar[dl]^{\be}&&*+={\circ}&&
*+={\circ}&\dots\\
\dots&*+={\circ}&&*+={\circ}&&*+={\circ}&&&&&&
}
}
$$

\vspace{2mm}

All the finite-dimensional string modules and the modules $C(w)$ for $w$ an infinite string are indecomposable pure-injectives
\cite{RinAC1}.

\subsection{Band modules; finite- and infinite-dimensional}\label{secband}\marginpar{secband}

To each band $b$ we associate a family of indecomposable modules.  The procedure is similar to that in defining string modules:
take an indecomposable $K[T,T^{-1}]$-module $L$, that is a vector space $V$ with an automorphism which is the action of $T$, deposit
the vector space $V$ at the first vertex and also at each vertex visited going once round the band $b$, then link these by identity
 maps (for the actions of the arrows) until the last step, where we link the last-deposited copy of $V$ to the first-placed one
 using the action of $T$.  This is similar to the procedure used to define string modules but the vector spaces placed at vertices
 of a walk might have dimension $>1$ and we ``join the ends of the string together with a twist".  This {\bf band module} (or $b$-band
  module if we wish to be more precise), denote it $M(b,L)$, is indecomposable and, if we replace the band by any cyclic permutation
  of it or by its inverse, then we obtain a module which is either isomorphic to this or to the module we would get if we started
  with the $K[T]$-module given by $V$ with the action of $T^{-1}$ (the difference being whether the arrow which carries the action
  of $T$ faces in the same or opposite direction).  The indecomposable finite-dimensional $K[T,T^{-1}]$-modules are factors of the
  ring by powers of maximal ideals, and are $n$-fold self-extensions of simple modules, so the indecomposable finite-dimensional
  modules supported on a band are parametrised by ${\rm maxspec}(K[T,T^{-1}]) \times {\mathbb N}$ where ${\mathbb N}$ denotes the
  set of positive integers.

All the indecomposable finite-dimensional modules have now been listed and there are no cases of isomorphism between them other
than those already mentioned.  The morphisms between these modules also have been completely described in combinatorial terms,
see Section \ref{secmor}.

The description that we have given of the modules associated to a band $b$ is essentially the description of a functor from
$K[T, T^{-1}]\mbox{-}{\rm Mod}$ to $R\mbox{-}{\rm Mod}$.  Up to equivalence, there are two such functors, the difference between
them being the orientation of the arrow where we put the twist by $T$.  It does not matter which we use since they have the same
image.   Any such functor is a representation embedding in the strong sense that it preserves indecomposability and reflects
isomorphism (`strong' because the usual definition asks this just for finite-dimensional modules) and hence induces a homeomorphic
embedding of Ziegler spectra.

\begin{theorem}\label{repemb}\marginpar{repemb} (see \cite[5.5.9]{PreNBK}) If $F:R\mbox{-}{\rm Mod} \rightarrow S\mbox{-}{\rm Mod}$
is a representation embedding (in the above strong sense) then $F$ preserves pure-injectivity and induces a homeomorphic embedding
of $_R{\rm Zg}$ as a closed subset of $_S{\rm Zg}$.
\end{theorem}

We can deduce the following description of the infinite-dimensional indecomposable pure-injective band modules.  To each band (up
to equivalence) $b$ and simple $K[T,T^{-1}]$-module $S$ we have its image, $M(b,S)$ a {\bf quasisimple} module.  We denote the
$n$-fold self extension of this module (which is the image of the $n$-fold self extension of $S$) by $M(b,S[n])$.  There is a ray
of irreducible monomorphisms starting at that module and the direct limit along this ray is the {\bf Pr\"ufer module} $\Sigma (b,S)$.
Dually, there is the inverse limit of the coray of irreducible epimorphisms ending at that quasisimple - the {\bf adic} module
$\Pi(b,S)$.  Finally there is a unique generic module $G= G_b$ associated to the band $b$ - obtained by applying the above process
to the $K[T,T^{-1}]$-module $K(T)$.  Each of these is an indecomposable pure-injective $R$-module which is the image, under a representation embedding as above, of the corresponding indecomposable pure-injective $K[T,T^{-1}]$-module.  See \cite{C-BTrond, PreTame, RinTame} and \cite[\S 8.1]{PreNBK} for more on these modules.

Theorem \ref{repemb} also applies with $R$ being the $K$-path algebra of some orientation of $\widetilde{A_n}$, where $n$ and the
orientations of the arrows of $\widetilde{A_n}$ are chosen so that it is the obvious intermediary between $K[T, T^{-1}]\mbox{-}{\rm Mod}$
and the $b$-band modules in $R\mbox{-}{\rm Mod}$.

The relation between the structure of the category $K[T,T^{-1}]\mbox{-}{\rm mod}$ and the infinite-dimensional points of
${_{K[T,T^{-1}]}{\rm Zg}}$, as well as the description of that space, is well-known; similarly for $K\widetilde{A_n}$.  We will
make considerable use of this knowledge, which can be found in, for example, \cite[\S 8.1.2]{PreNBK}.

\subsection{The list of modules}\label{seclist}\marginpar{seclist}

In summary, this is the complete (by \cite{P-P14}) list of types of indecomposable pure-injective modules over a domestic string
algebra, with parametrising sets as described earlier (we follow the notation of \cite{RinAC1}):

\begin{enumerate}
\item finite-dimensional string modules;
\item finite-dimensional band modules;
\item infinite direct sum string modules;
\item infinite direct product string modules;
\item infinite mixed string modules;
\item Pr\"{u}fer modules;
\item Adic modules;
\item Generic modules.
\end{enumerate}

\noindent We will group these modules in the following way:
1 and 3-5 will be referred to as {\bf string} modules; 2 and 6-8 will be referred to as {\bf band} modules.

We will fix the following notation throughout the rest of the paper:\begin{itemize}
\item $A(x)$ will be used for any indecomposable pure-injective module associated to the string or band $x$;
\item $B(x)$ will be used for any indecomposable band-type module associated to the band $x$;
\item $C(x)$ will denote the indecomposable pure-injective string-type module associated to the string $x$.
\item We will use the notation $M(x)$ only to refer to the direct sum module associated to the string $x$.
\end{itemize}

\section{The morphisms}\label{secmor}\marginpar{secmor}

\subsection{Graph maps}\label{secgraph}\marginpar{secgraph}

A complete description of the maps between string and band modules has been developed by Krause and Crawley-Boevey in a series of
papers \cite{Kra91}, \cite{C-BTrond} and \cite{C-Bzero-rel}.  We recall this briefly.  Let $u$ be a (finite or infinite) string.
Suppose that $v$ is a substring of $u$ which is closed under successors in $u$, meaning that at its endpoint(s) in $u$ the next
arrow(s), if there are any, point to, rather than away from, the end-vertices of $v$; we say that $v$ is an {\bf image substring}
of $u$.  Clearly this occurrence of $v$ as a substring of $u$ gives a submodule of $C(u)$ which is the image of an embedding
$C(v) \rightarrow C(u)$.  For instance, if $u = \al \ga \al^{-1}\be^{-1}$ is a string over $X_1$ then $\al$ is an image substring
of $u$ and $\al^{-1}$ is not.

Dually, if $v$ is closed under predecessors in $u$, meaning that at its endpoint(s) in $u$ the next arrow(s), if any, point away
from the end-vertices of $v$ - we say that $v$ is a {\bf factor substring} of $u$ - then this occurrence of $v$ as a substring of
$u$ defines a natural epimorphism from $C(u)$ to a copy of $C(v)$.  For the example $u$ just above, $\al^{-1}$ is a factor substring,
whereas $\al $ is not.

A {\bf graph map} from the string module $C(u)$ to the string module $C(u')$ is any map which is a composition of such an epimorphism
with a monomorphism of the sort described, so it is obtained by matching a factor substring of $u$ with an image substring of $u'$.
Then (see the references) any morphism between string modules is a finite linear combination of graph maps.  Here is an example over $X_1$.

$$
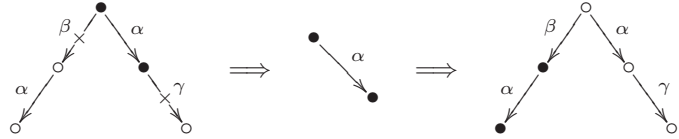
\parbox{4cm}{%
\xymatrix@C=10pt@R=16pt{%
&&*+={\bullet}\ar[dl]|-{\times}_{\be}\ar[dr]^{\al}&&\\
&*+={\circ}\ar[dl]_{\al}&&*+={\bullet}\ar[dr]|-{\times}^{\ga}&\\
*+={\circ}&&&&*+={\circ}
}}
\quad \Longrightarrow \quad
\parbox{4cm}{%
\xymatrix@C=16pt@R=16pt{%
*+={\bullet}\ar[dr]^{\al}&\\
&*+={\bullet}
}}
\quad \Longrightarrow \quad
\parbox{4cm}{%
\xymatrix@C=10pt@R=16pt{%
&&*+={\circ}\ar[dl]_{\be}\ar[dr]^{\al}&&\\
&*+={\bullet}\ar[dl]_{\al}&&*+={\circ}\ar[dr]^{\ga}&\\
*+={\bullet}&&&&*+={\circ}
}}
$$

Morphisms to and from band modules have a similar description.  Given a band module $B = M(b,S[n])$, a {\bf graph map} to or from $B$
is described combinatorially (and algebraically) similarly to the above, where we allow the string associated to $B$ to be a large
enough power of $b$ (thus, a suitable-length, depending on the other module, segment of $^\infty b^\infty$).  Then any morphism
between $B$ and a string module or a module associated to a different band (including $b^{-1}$) is a linear combination of graph
maps.  The maps between $B$ and modules associated to the same band are linear combinations of graph maps and maps which are
compositions of irreducible maps (as seen in the Auslander-Reiten quiver).

\begin{example}
Consider the Kronecker algebra $\xymatrix{*+={\circ}\ar@/^0.4pc/[r]^{\al}\ar@/_0.4pc/[r]_{\be} & *+={\circ}}$ and the band
$b = \beta\alpha^{-1}$.  If $L$ is an indecomposable finite-dimensional $K[T,T^{-1}]$-representation, then the corresponding band
module is given by the following, where $V$ is the underlying vector space of
$L$  \[ \xymatrix{ V \ar@/^/[r]^{\alpha =T} \ar@/_/[r]_{\beta = 1_V} & V}. \] Consider the string module $M(v)$ where
$v = \beta\alpha^{-1}\beta\alpha^{-1}$.  Fix some nonzero $g\in {\rm Hom}_K(K,V)$, equivalently element of $V$ and map the left end
of the string module $M(v)$ to $M(b,L)$ to this element.  This then determines a graph map as shown below, which will be an
embedding iff the dimension of $V$ is at least 3.

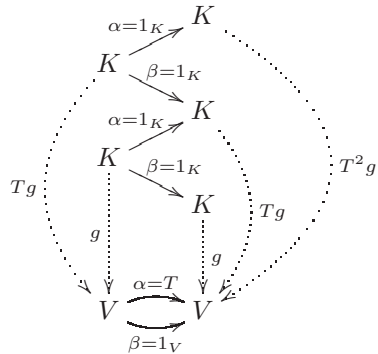
\[ \renewcommand{\labelstyle}{\scriptstyle} \xymatrix@C=20pt@R=5pt{
                                   & K \ar@{.>}@/^4pc/[ddddddd]^{T^2g} \\
K \ar[ur]^{\alpha = 1_K} \ar[dr]^{\beta = 1_K} \ar@{.>}@/_2pc/[dddddd]_{T g}     &                                      \\
                                   & K \ar@{.>}@/^1.5pc/[ddddd]^{T g}    \\
K \ar[ur]^{\alpha = 1_K} \ar[dr]^{\beta = 1_K} \ar@{.>}[dddd]_g     &                             \\
                                   & K \ar@{.>}[ddd]^{g}                              \\
                                   &                                \\
                                   &                                \\
V \ar@/^/[r]^{\alpha =T} \ar@/_/[r]_{\beta = 1_V} &     V                          \\
                                   &
}\]

\end{example}

Here is an example of an endomorphism of a band module which is not the image of a morphism under one of the standard representation
embeddings from $\widetilde{A_n}$-modules.  In the following diagram of $X_1$-modules, the first map is the obvious epimorphism and
the second is the inclusion.

$$
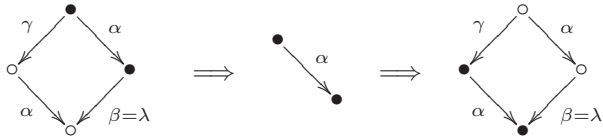
\parbox{4cm}{%
\xymatrix@C=16pt@R=16pt{%
&*+={\bullet}\ar[dl]_{\ga}\ar[dr]^{\al}&\\
*+={\circ}\ar[dr]_{\al}&&*+={\bullet}\ar[dl]^{\be=\lam}\\
&*+={\circ}&
}}
\quad \Longrightarrow \quad
\parbox{4cm}{%
\xymatrix@C=16pt@R=16pt{%
*+={\bullet}\ar[dr]^{\al}&\\
&*+={\bullet}
}}
\quad \Longrightarrow \quad
\parbox{4cm}{%
\xymatrix@C=16pt@R=16pt{%
&*+={\circ}\ar[dl]_{\ga}\ar[dr]^{\al}&\\
*+={\bullet}\ar[dr]_{\al}&&*+={\circ}\ar[dl]^{\be=\lam}\\
&*+={\bullet}&
}}
$$

\section{Duality}\label{secdual}\marginpar{secdual}

We will make use of a couple of dualities.  One is the duality $(-)^*= {\rm Hom}_K(-,K)$ between categories of finite dimensional
left and right $R$-modules.  We need the following observations about the action of this on certain modules. Suppose that $b$ is a
band and that $M=M(b,S)$ is a quasisimple module.  Consider the $\Hom$-dual $M^*$ of $M$.  One can check that $M^*$ is the
quasisimple right module over $R$ corresponding to the same ``band'' $b$ (of course all the arrows are now reversed) with, depending
on how one chooses to parametrise these right band modules, the same or inverse value of the parameter $\lam$.

The other duality that we will use is elementary duality $D$, which is essentially the Auslander--Gruson--Jensen duality
(see \cite[Sec. 10.3]{PreNBK}) between the categories of finitely presented functors but which induces dualities on other structures.
For instance (see \cite[\S 1.3]{PreNBK}) $D$ provides an anti-isomorphism between the lattices of right and left pp-formulas over $R$,
taking a pp formula $\phi$ to its dual $D\phi$.  Furthermore (see \cite[\S 5.4]{PreNBK}) the map $(\phi/\psi)\mapsto (D\psi/D\phi)$
between basic open sets extends to open (and closed) sets of the right and left Ziegler spectra of $R$ and, as such, defines an
isomorphism between the lattices of open (respectively closed) subsets of these Ziegler spectra which preserves infinite unions
(resp.~infinite intersections).  If, for example, KG-dimension is defined then it given a bijection $N\mapsto DN$ between points of the right and left Ziegler spectra (see \cite[5.4.20]{PreNBK}).

In the case of domestic string algebras, from the non-existence \cite[5.6]{P-P14} of superdecomposable pure-injectives it follows,
\cite[5.3.24, 5.4.12, 5.4.18]{PreNBK}, that the points of the right and left Ziegler spectra are paired up,
$N\rightarrow DN \rightarrow D^2N=N$, by elementary duality, which is a homeomorphism between these spaces and which, restricted to
finite-dimensional indecomposables, is the usual duality ${\rm Hom}_K(-,K)$.  In the case of $K[T]$-modules, the elementary dual of
a Pr\"ufer module is the adic module with the same parameter, and this coincides with its dual using, say, the duality
${\rm Hom}_{\mathbb Z}(-,{\mathbb Q}/{\mathbb Z})$.  The latter duality applied to an adic module gives the direct sum of the
corresponding Pr\"ufer module with many copies of ${\mathbb Q}/{\mathbb Z}$ but the elementary dual of the adic is identifiable as
the unique isolated point in the corresponding Ziegler-closed set (the Ziegler-closure of the adic and of the Pr\"ufer are homeomorphic).
Essentially the same applies over tame hereditary algebras.  Simply from the topological information, the generic left and right modules
are dual to each other.

General immediate consequences of there being the duality homeomorphism are that it restricts to a homeomorphism between the
Ziegler-closure of any point and that of its dual and that a basis of open neighbourhoods of any point dualises to a basis of open
neighbourhoods of its dual.

\section{Rank and dimensions}\label{secrkdim}\marginpar{secrkdim}

\subsection{Cantor--Bendixson rank}\label{seccbrk}\marginpar{seccbrk}

The Cantor--Bendixson analysis of any space $T$ runs as follows. At the first step we identify the isolated=open points of this space,
give them CB rank 0, and then remove them, leaving the first derivative, $T'$, of $T$.  In the case of the Ziegler spectrum of an
artin algebra, the isolated points are exactly the finite-dimensional indecomposables (see \cite[5.3.33]{PreNBK}).  The input to the
next step is $T'$, the points isolated in $T'$ are assigned CB rank 1 and, inductively, we continue the procedure of forming
derivatives and assigning ranks to points, transfinitely, by taking intersections at limit stages.  If this process stabilises with a
non-empty subspace (which must have no isolated points) then we assign CB-rank $\infty$ to these remaining points and say that the
$\CB$-rank of $T$ is $\infty$ or undefined.  Otherwise there is a least ordinal $\lambda$ where we reach the empty space, and which
is not a limit if $T$ is compact, in which case we set the CB-rank of $T$ to be $\lambda -1$.  In the case of domestic string algebras
all the points of this maximal rank will, since we have the condition (IC) (see \cite[\S 5.3.2]{PreNBK}), be of finite length over their endomorphism rings
(see \cite[5.3.22]{PreNBK}) - hence generic in Crawley-Boevey's terminology, provided that $R$ is not of finite representation type
(recall that a {\bf generic} module is one of finite endolength which not of finite length).

\section{Pp formulas}\label{secpp}\marginpar{secpp}

Some of our arguments use pp formulas and the lattice they form.  Exposition and summary of this material is easily available from
various of the background references so we just explain some of the more technical details here.

Suppose that $s$ is a vertex of $Q$ and let $e_s\in R$ be the corresponding primitive idempotent.  If $d \in H_{s,1}$ is a string then the property of being divisible by $d$ (that is $x\in dM$) can be expressed by a pp formula and includes the condition $x=e_sx$.  A somewhat
stronger condition is defined as follows.  If there is an arrow $\ga$ such that $d\ga^{-1}$ is a string
then the formula $(.d)(x)$ states that $x= e_s\, x$ and $x\in d\ga^{-1}(0)$; otherwise $(.d)(x)$ say just that $x= e_s x$ and
$x\in d M$. If $c$ a string in $H_{s,-1}$ then the corresponding formulas $(.c)$ are defined similarly.

If $c\in H_{s,-1}$, $d\in H_{s,1}$, then $(c^{-1}.d)$ is the conjunction of $(.d)$ and $(.c)$.
The special case $(1.d)$ which, on the left, is imposing the condition that $\zeta x=0$ if there is an $\zeta$ such that $\zeta d$ is
a string; $(^+1.d)$ which, on the left is imposing the condition that $x$ is divisible by the arrow $\zeta$ such that $\zeta^{-1}d$ is
a string, if there is such a $\zeta$ (if there is no such $\zeta$ we make $(^+1.d)$ the condition $x=0$).  See \cite{P-P04} or
\cite[Sec. 5.7]{Har}, for instance for these notations.

If $c$ is a nonempty string which does not begin on a peak (that is, there is a string of the form $\gamma^{-1}c$)
then we define $^+c$ to be the result of adding a hook to the left of $c$, that is, $^+c=c_1\gamma^{-1}c$ where $c_1$ contains only
direct letters and is maximal possible such.  Similarly $c^+$ is defined by adding a hook on the right, if possible
(see \cite[p.~162]{B-R}).

To any nonzero element $m\in e_sM$ of a module element one may (see \cite[p.~6]{P-P14}) associate a (1-sided)
string $v= v(m)$ describing the divisibility of $m$ by strings in $H_{s,1}$.  Similarly a string $u= u(m)$ describes the divisibility
of $m$ by strings in $H_{s,-1}$, and these are combined into a 2-sided string $w(m)= u^{-1}.v$ (as mentioned earlier, informally we
consider divisibility by $v$ as divisibility to the right and divisibility by $u$ as divisibility to the left).

\subsection{pp formulas and m-dimension}\label{secmdim}\marginpar{secmdim}

Suppose that $L$ is a modular lattice; we have in mind the lattice of pp formulas in one free variable.  At the first stage of the
m-dimension analysis of $L$ we factor $L$ by the congruence relation generated by the intervals of finite length:  points $a$ and
$b$ will be identified iff the interval $[a+b,a\wedge b]$ is of finite length (here $+$ denote the sup operation).  The result is
again a modular lattice, and we repeat the process, transfinitely, where at limit stages we factor $L$ by the union of the inverse
images in $L$ of the congruence relations generated so far.  If the procedure stabilises with a nontrivial lattice (necessarily
densely ordered) then we say that the {\bf m-dimension} of $L$ is $\infty$ or undefined.  Otherwise, provided $L$ has a top and a
bottom, the first ordinal $\mu$ such that the $\mu$-collapse of $L$ is trivial is not a limit, and then we say that the {\bf m-dimension}
of $L$ is $\mu-1$.  In particular m-dimension 0 exactly means finite length.  See \cite[\S 7.2]{PreNBK} for details.

Clearly there is a rough parallel with the definition of Cantor--Bendixson rank.  There is also a parallel process, which uses
localisation, in the abelian category of finitely presented functors; see Section \ref{seckgdim}.  Under certain conditions
(see \ref{kgeqcb}), including that of there being no superdecomposable pure-injectives, the CB-analysis of the Ziegler spectrum
runs parallel with the m-dimension analysis on the lattice ${\rm pp}_R$ of all pp-formulas.

\subsection{Krull--Gabriel dimension, m-dimension and CB-rank}\label{seckgdim}\marginpar{seckgdim}

Krull--Gabriel (KG-) dimension is defined, in the first instance, on finitely presented functors from $R\mbox{-}{\rm mod}$ to
${\bf Ab}$ but it can be extended to give a rank on points.  In fact, the m-dimension and KG-dimension of an indecomposable
pure-injective $N$ are equal (\cite[4.2.8]{BurThes}).  In our case both these dimensions also coincide with CB-rank so there is
no need to include a definition of KG-dimension here.

We recall just a little about m-dimension from \cite{PreBk}, based on \cite{Zie} (which in turn refined Garavaglia's elementary
Krull dimension \cite{Gardim}).  We define (see \cite[p.~214]{PreBk}) the {\bf m-dimension} of an indecomposable pure-injective $N$
in terms of m-dimension of intervals in the lattice of pp formulas which are open on $N$:
${\rm mdim}(N)= {\rm min} \{ {\rm mdim}[\psi,\phi]: \phi(N)>\psi (N)\} = {\rm min} \{ {\rm mdim}[\psi,\phi]: N\in (\phi/\psi)\}$.
Also let $p$ be any non-zero pp-type realised in $N$ and set
${\rm mdim}(p)= {\rm min} \{ {\rm mdim}[\psi,\phi]: \phi \in p^+, \psi\in p^- \}$.

\begin{theorem}\label{mdimtp}\marginpar{mdimtp} (\cite[8.7]{Zie}, see \cite[10.23]{PreBk}) If $p$ is the pp-type of a non-zero
element of $N\in \,_R{\rm Zg}$ then ${\rm mdim}(p) ={\rm mdim}(N) $.
\end{theorem}

\begin{prop}\label{kgeqmdim}\marginpar{kgeqmdim} \cite[4.2.8]{BurThes} If $N$ is an indecomposable pure-injective then
${\rm KGdim}(N) = {\rm mdim}(N)$.
\end{prop}

If the isolation condition (IC) holds (as it does for modules over domestic string algebras since, by \cite[5.6]{P-P14}, that there
is no superdecomposable pure-injective) then this equals the CB rank of $N$.

\begin{theorem}\label{kgeqcb}\marginpar{kgeqcb} (\cite[8.6]{Zie}, see \cite[10.19]{PreBk}) (assuming IC) For $N\in \,_R{\rm Zg}$
we have ${\rm mdim}(N) = {\rm KGdim}(N) = {\rm CB}(N)$.
\end{theorem}

In fact, all this holds true when relativised to any closed subset $X$ of the Ziegler spectrum and applies to any $N\in X$, with
the lattice of pp formulas being replaced by its quotient which is obtained by identifying formulas which agree on all modules in
$X$, and with CB rank being measured in $X$.  Then, still under the assumption IC, if ${\cal U}$ is a basis of basic open
neighbourhoods of $N$ in $X$, ${\rm mdim}_X(N)= {\rm min} \{ {\rm mdim}_X[\phi,\psi]: (\phi/\psi)\in {\cal U}\} = {\rm CB}_X(N)$,
where the subscripts indicate these relativisations (\cite[p.~214, Exer.~3]{PreBk}).

\begin{cor}\label{m-conn}\marginpar{m-conn}
Let $R$ be a domestic string algebra.

(1) Let $\phi > \psi$ be a pair of pp-formulas.  The m-dimension of the interval $[\psi; \phi]$ in the lattice of pp formulas equals
the maximum of CB-ranks of points in the corresponding open subset $(\phi/\psi)$ of the Ziegler spectrum $_R{\rm Zg}$.

(2) Suppose that $N$ is an indecomposable pure injective $R$-module with a chosen neighbourhood basis ${\cal U}$ of open sets. Then
the CB-rank of $N$ is the minimum of the m-dimensions of intervals $[\psi; \phi]$ where the pairs of pp-formulas $(\phi/\psi)$ run
over ${\cal U}$.
\end{cor}

\section{Computing rank from the bridge quiver}\label{secrecp}\marginpar{serecp}

Here we compute the Cantor--Bendixson rank of each point of $_R{\rm Zg}$, $R$ a domestic string algebra, in a way that makes natural
links between approximation through the bridge quiver and topological approximation.  Since approximation by direct (and sometimes
inverse) limits does imply topological approximation (see \cite[3.4.7, p.~114]{PreNBK}) we can see some of this algebraically as
well.  Because there are no superdecomposable pure-injectives, the CB-rank of each point equals its KG dimension which equals its
m-dimension (\ref{m-conn}(2)); sometimes, for brevity, we will refer just to the `rank' of a point.  The arguments we use in this
section have parallels which refer to the lattice of pp formulas and/or to the factorisation structure of morphisms in
$R\mbox{-}{\rm mod}$, indeed, any of these three approaches can be used to establish the `recipes' we give here for computing the
rank of a point.  The approach using topological approximation perhaps involves less technical detail since some of that gets hidden
in the topology and the appeals to compactness.

The key to computing the rank of the module $A(w)$ associated to $w$ is to consider whether the corresponding path in the bridge
quiver is maximal or whether, and how far, it could be extended in either direction.  When computing ranks of band modules we will
 need more detailed information, specifically whether paths in the bridge quiver leave a band {\it via} an arrow pointing away from
 or into the band.  Note that it is possible (for example over $X_3$) for a path to include both a band and its inverse (but recall
 that for domestic algebras there are no cyclic paths in the bridge quiver).

As well as computing ranks, we give a neighbourhood basis for each infinite-dimensional point of the Ziegler spectrum of $R$.

\subsection{Representation embeddings and topology}

In computing these ranks we make heavy use of the fact that if $N$ is a point in a topological space and $U$ is an open set containing
$N$ then the CB rank of $N$ in the whole space is equal to the CB rank of $N$ computed in the relative topology on $U$.  This allows
us to ignore some modules when computing ranks.

Recall that the finite-dimensional points of $_R{\rm Zg}$ are exactly those of rank 0.  We ignore these finite-dimensional points
from now on and work in the (closed) subset consisting of infinite-dimensional points.  Since that set is the first Cantor--Bendixson
derivative, $_R{\rm Zg}'$, of $_R{\rm Zg}$, we can compute ranks in this set and then just add 1 to get the rank of an
infinite-dimensional point in the whole space.  For 1-sided strings, the computation of rank is already in \cite[\S 7]{P-P04}, but
we will recompute these here, on our way to the ranks of other points.

Next, for each band $b$ choose a representation embedding from $K[T,T^{-1}]\mbox{-}{\rm Mod}$ to $R\mbox{-}{\rm Mod}$ as in
Section \ref{secband}; this induces a homeomorphic embedding of $_{K[T,T^{-1}]} {\rm Zg}$ onto the closed subset of $_R{\rm Zg}$
(\cite[Thm.~7]{PreRepn}) consisting of $b$-band modules.  Note that the images of these embeddings for different bands are disjoint.

Recall that we can alternatively use the path algebra of a suitably
oriented quiver $\widetilde{A_n}$ in place of $K[T]$ (restricting to the closed subset of points where the action at each arrow of $\widetilde{A_n}$ acts
invertibly).  There are just finitely many bands, so the union over all bands $b$ of the images of these induced homeomorphic
embeddings is closed and consists of all the indecomposable pure-injective band modules.  Denote the complement of this, the set
of all string modules, by $U_1$; this then is an open subset of $_R{\rm Zg}$ and working in $U_1$ will allow us to ignore all
band modules when computing ranks of string modules.  It will also simplify consideration of band modules since we can fix one band $b$
and then work in the open set consisting of the $b$-band modules and the string modules; that is, we can ignore the modules associated
to other bands.

\subsection{Ranks of string modules}\label{secrkstring}\marginpar{secrkstring}

In the remainder of this section we will work in $U=U_1\, \cap\, _R{\rm Zg}'$ and ``open set'' will mean open subset of this set
(recall that $ _R{\rm Zg}'$ is the set of infinite-dimensional point of $ _R{\rm Zg}$).

\vspace{4pt}

We show that the rank of an infinite-dimensional string module $C(w)$ is given in terms of the position in the bridge quiver of the
infinite band(s) at the end(s) of $w$.

We recall the neighbourhood bases that were computed in \cite{P-P04} for 1-sided strings and in \cite{P-P14} for 2-sided strings.

\begin{theorem}\label{basis2}\marginpar{basis2} (\cite[6.2]{P-P14})
Suppose that $w=u^{-1}v$ is a 2-sided non-periodic string over a domestic string algebra and $C(w)$ is the corresponding
indecomposable pure injective module. A basis of open sets in the Ziegler topology for $C(w)$ is given by the
pairs $(c^{-1}.d) \,\, / \,\, (e^{-1}.d)+ (c^{-1}.f)$, where $c\leq u< e$ in $\wh H_{-1}$ and $d\leq v< f$ in $\wh H_1$.
\end{theorem}

In the case that the infinite word $w$ is 1-sided, say $w=c^{-1}.v$ where $c$ is finite, then, see Section \ref{secpp}, $c$ has a
successor, $c^+$ in $H_{-1}$; we will write $^+1$ for the inverse of this successor in the case that $c$ is empty (rather, $c=1_{-1}$).
The next result is stated in \cite{P-P04} for general finite $c$ in place of $1$.

\begin{theorem}\label{basis1}\marginpar{basis1} (\cite[5.3]{P-P04})
Suppose that $v$ is a 1-sided non-periodic string over a domestic string algebra and $C(v)$ is the corresponding
indecomposable pure injective module. A basis of open sets in the Ziegler topology for $C(v)$ is given by the
pairs $(1.d) \,\, / \,\, (^+1.d)+ (1.f)$, where $d\leq v< f$ in $\wh H_1$.
\end{theorem}

Making good choices in these results, we can get nicely described neighbourhood bases.

\begin{cor}\label{stringnbhd}\marginpar{stringnbhd} Let $w=\, ^\infty b_0cb^\infty$ be a 2-sided string denote by $w_n$ its image
substring $^n b_0cb^n$ and let $U_n$ be the set of string modules $C(y)$ such that $y$ contains $w_n$ as an image substring.  Then
the $U_n$ form a basis of open neighbourhoods of $C(w)$.

Similarly if $w=cb^\infty$ with $b$ a band and $c$ finite then let $w_n =cb^n$ and let $U_n$ be the set of (1-sided) string modules
$C(y)$ such that $y$ contains $w_n$ as an initial image substring.  Then the $U_n$ form a basis of open neighbourhoods of $C(w)$.
\end{cor}
\begin{proof}  Split $w_n$ as $u_nv_n$ where $u_n =b_0^n$ and $v_n=cb^n$.  Denote by $v_n^\circ$ the word obtained from $v_n$ by removing the last (inverse) letter and by $^\circ u_n$ denote
the word obtained from $u_n$ by removing the first (direct) letter.  Then the open set $\big((u_n.v_n)/(^\circ u_n.v_n + u_n.v_n^\circ)\big)$ has the form of a set as in \ref{basis2}.  This open set separates $C(w)$ from all band modules; furthermore it
consists of the string modules $M(w')$ where $w'$ contains $w_n$ as an image subword.  For, if
$C(w') \in \big((u_n.v_n)/(^\circ u_n.v_n + u_n.v_n^\circ)\big)$ then there is a nonzero graph map from $M(w_n)$ to $C(w')$ which does not factor through
the canonical map from $M(w_n)$ to the string module of any truncation.  The first says that some factor substring of $w_n$ occurs
an image substring of $w$; the second says that there is such an occurrence where it is $w_n$, rather than a proper image string,
which occurs.

The 1-sided case is similar, the relevant open set being $(1.w_n/^+1.w_n + 1.w_n^\circ)$.
\end{proof}

Note that, given any 1-sided infinite string module any of the open neighbourhoods in \ref{basis1} serves to separate it from all
2-sided strings.

Now we compute the ranks of points.

\begin{theorem}\label{1stringrank}\marginpar{1stringrank} If $w=cb^\infty$ is a 1-sided right-infinite string whose terminal band
has right indent $t$ then the CB rank of $C(w)$ is $t+1$.
\end{theorem}
\begin{proof} Consider the intersection of the open set $U_n$ from \ref{stringnbhd} with $U$; this consists of the modules $C(x)$
where $x$ is a 1-sided (infinite) string with $w_n$ as an initial image substring.  If the right indent of $b$ is $0$ then the only point in
this open set is $C(w)$ itself which, therefore, has CB-rank 0 in $U_n\cap U$, hence has CB-rank 1 in the whole space.  If the right
indent of $b$ is $ 1$, so there are strings of the form $b^nu'b_1^\infty \neq b^\infty$, but no strings beginning with $b$ and
giving longer paths in the bridge quiver, then, by the case we just dealt with, each string module of the form $C(cb^nu'b_1^\infty)$
belongs to the open set $U'$ and has rank 1.  There are no other points in that open set apart from $C(w)$ which, therefore,
has rank 2:  it cannot have rank 1, since then we would have a (Ziegler-basic, hence) compact open set being a union of infinitely
many isolated points - contradiction.  This argument continues inductively on the right indent of the band $b$ (the cases 0 and 1
having just been treated) and so we obtain the result.
\end{proof}

Now we move on to the 2-sided strings.

\begin{theorem}\label{2stringrank}\marginpar{2stringrank} If $w$ is a 2-sided infinite string whose terminal bands have indent $s$
and $t$ respectively then the CB rank of $C(w)$ is $s+t+2$.
\end{theorem}
\begin{proof}  Say $w= \,^\infty b_0ub_1^\infty$.  Again take the open set $U_n$ from \ref{stringnbhd} and consider its intersection
with $U$.  As before, we start with the case that the corresponding path from $b_0$ to $b_1$ is maximal in the bridge quiver, and then
 work inwards.

In the case that the path between $b_0$ and $b_1$ is maximal, the only points in $U_n\cap U$, apart from $C(w)$, are the (infinitely
 many) one-sided infinite string modules which contain $b_0^nub_1^n$ as an image substring.  By \ref{1stringrank}, these all have
 rank $1$ so, arguing as in the 1-sided case, $C(w)$ has rank $2$.

We continue inductively, first working in from the right side.  Suppose that $w$ has left indent $0$ and right indent $t$.  Then the
set $U_n \cap U$ consists of: 1-sided string modules of the form $^\infty b_0ub_1^nv$ with $v$ a finite string - these have rank $1$
by \ref{1stringrank}; 1-sided string modules of the form $vb_0^nub_1^\infty$ with $v$ finite - these have rank $\leq t+1$ by
\ref{1stringrank}; and the two-sided string modules where the string has the form $^\infty b_0ub_1^n\dots$ and has right indent
$\leq t-1$.  Note that the value $t-1$ is achieved among these 2-sided string modules; moreover, by induction, these 2-sided string
modules have rank $\leq t-1+2 = t+1$, with the value $t+1$ being achieved.  Since $t\geq 1$ there are infinitely many points in
$U_n\cap U$ with (maximal) rank $t+1$ (because we have modules of the form $^\infty b_0ub_1^k\dots$ for all $k\geq n$).  Removing
all these points from $U'$ leaves only $C(w)$ which therefore has rank $t+2$.

Now we can also move in from the left side:  suppose first that the left indent, $s$, of $w$ is $1$ and the right indent is $t$.
Then the 1-sided strings in $U_n\cap U$ have ranks up to and including ${\rm max}\{s=1,t\}+1$.  The 2-sided strings have, by the
above, ranks up to and including $t+2$ - and there are infinitely many of these with rank $t+2$ - and then there's $C(w)$ which,
therefore, has rank $1+t+2 = s+t+2$.  Working in from the left as we did on the right but inducting on the left indent, we see
that a 2-sided string with left indent $s$ and right indent $t$ has rank $s+t+2$.
\end{proof}

\begin{example}  For instance, let $R= \Lam_2$ and let $w= (\eps\del^{-1})^{\infty}$. By choosing $b_1= \eps\del^{-1}$ and
$b_2= \al\be^{-1}$ we see that there is a path $b_1\to b_2$ starting with $b_1$ in the bridge quiver and there are no others,
so 1 is the maximal length of such paths and the indent of $b_1$ is 1. It follows that the CB-rank of the corresponding
indecomposable pure-injective, the direct product module $C(w)$, is $1+1= 2$. On the other hand there is no bridge starting
from $b_2$, hence the direct product module corresponding to the string $b_2^{\fty}$ has CB-rank $1$.  Each terminating band
of the 2-sided string $u=\, ^\infty(\eps\del^{-1}) \eps \ga  (\al\be^{-1})^\infty$ has indent 0, so the corresponding mixed
module $C(u)$ has CB-rank $0+0+2=2$.
\end{example}

\subsection{Ranks of band modules}\label{secrkband}\marginpar{secrkband}

\subsubsection{Pr\"ufer modules}

Let $b= \al \dots \be^{-1}$ be a band and let $\Sigma(b,S)$ be a $b$-Pr\"ufer module, with corresponding quasisimple module $M(b,S)$;
set $B_i=M(b,S[i])$.  Fix an element $m\in \be B_1 \cap \al B_1$ in the socle of $B_1$.  Consider the sequence of irreducible
embeddings $B_i \rightarrow B_{i+1}$ and choose $\phi_i$ to be any pp formula which generates the pp-type of the image of $m$ in
$B_i$; so we have a descending sequence $\phi_1>\phi_2 >\dots$ of pp formulas which together generate $p^+$ where $p$ is the pp-type of $m$ in $\Sigma(b,S)$.  Note that each interval $[\phi_{i+1}, \phi_i]$ is simple and that any pp formula strictly below
$\phi_i$ is less than or equal to $\phi_{i+1}$.  To see that, let $\theta<\phi_i$.  Then there is a morphism $f:B_i \rightarrow N$
 where $(N,n)$ is a free realisation of $\theta$.  By strictness of the inclusion, $f$ is not a split embedding, hence it factors
 through the almost split map $g: B_i \rightarrow B_{i+1} \oplus B_{i-1} $ (the second term being absent if $i=1$), say $f=hg$.
 Note that the image of $m$ in the second component $B_{i-1}$ is zero, so $h$ takes $m\in B_{i+1}$ to $n$ and hence
 $\phi_{i+1} \geq \theta$, as claimed.

Let $c$ be a cyclic permutation of $b$ which starts with an inverse arrow and ends with a direct arrow.  Write $M_j$ for the
string module $M(c^j)$ (so the $M_j$ are the images of preinjective $\widetilde{A_n}$-modules under a suitable representation
embedding).  In each $M_j$ let $n_j$ be the element in $\al M_j \cap \be M_j$ which is the image of the composition
$B_1 \rightarrow \Sigma(b,S) \rightarrow M_j$, the maps being the natural ones. Let $\psi_j$ be any pp formula which generates
the pp-type of $n_j$ in $M_j$.  Consider the epimorphism $M_{j+1} \rightarrow M_j$ which is obtained by factoring out the
rightmost copy of $c$ and mapping $n_{j+1}$ to $n_j$; so we have the ascending chain $\psi_1 < \psi_2 < \dots$.  Also note that
$\phi_i > \psi_j$ for every $i,j$ (consider the morphism $B_i \rightarrow \Sigma(b,S) \rightarrow M_j$).

\begin{example}  To see all this more concretely, we can use that the direct product module $\ov M = \ov M(^{\fty}b ^{\fty})$
has, as a direct summand, every Pr\"ufer module (it is the dual of the image of $K[T,T^{-1}]$ under a suitable representation
embedding); there is a unique, up to scalar multiple, embedding of the regular module $B_1$ into this module (see \cite{Kra91},
or better \cite[Sec. 6.3.2]{Har}).  For example, take $b= \al\be^{-1}$ over $X_3$ (in fact everything is happening over the
Kronecker algebra) and let $m$ be as above with $B_1=M(\al\be^{-1},\lambda)$. Then $m$ maps to the indicated infinite sum of
elements in the socle of $\ov M$ and the pp-type of $m$ in $\ov M$ is $p$.

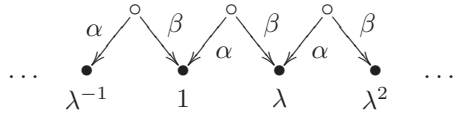
$$
\vcenter{%
\def\labelstyle{\displaystyle}
\xymatrix@C=12pt@R=16pt{%
&&*+={\circ}\ar[ld]_{\al}\ar[rd]^{\be}&&*+={\circ}\ar[ld]^{\al}\ar[rd]^{\be}&&
*+={\circ}\ar[ld]^{\al}\ar[rd]^{\be}&&\\
\dots&*+={\bullet}\ar@{}+<0pt,-10pt>*{\lam^{-1}}&&*+={\bullet}\ar@{}+<0pt,-10pt>*{1}&&
*+={\bullet}\ar@{}+<0pt,-10pt>*{\lam}&&*+={\bullet}\ar@{}+<0pt,-10pt>*{\lam^2}&\dots
}}
$$

\vspace{3mm}

Continuing with the notation above, any occurrence of $c^k$ as a factor substring of $^{\fty}b ^{\fty}$ gives a map of $\ov M$ onto the string module $M_k$. In the example $c$ must be $\be^{-1}\al$ and $M_2$ is shown below, with the (unique to scalar) image of $m$ being the sum of the indicated socle elements.  The formula $\psi_k$ is chosen to generate the pp-type of $n_k$ in $M_k$.

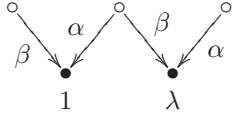
$$
\vcenter{%
\def\labelstyle{\displaystyle}
\xymatrix@C=14pt@R=18pt{%
*+={\circ}\ar[dr]_{\be}&&*+={\circ}\ar[dl]_{\al}\ar[dr]^{\be}&&*+={\circ}\ar[dl]^{\al}\\
&*+={\bullet}\ar@{}+<0pt,-10pt>*{1}&&*+={\bullet}\ar@{}+<0pt,-10pt>*{\lam}&
}}
$$

\end{example}

\begin{theorem}\label{prufernbhd}\marginpar{prufernbhd} The open sets $(\phi_i/\psi_j)$ defined above form a neighbourhood basis of open sets of the Pr\"ufer module $\Sigma(b,S)$.
\end{theorem}
\begin{proof}  With notation as above, by \cite[5.1.21]{PreNBK} a neighbourhood basis for $\Sigma(b,S)$ is given by sets of the form $(\phi/\psi)$ with $\phi\in p^+$ and $\phi > \psi\in p^-$.  We have seen already that we can take $\phi$ to be one of the $\phi_i$, so consider some pp formula in $p^-$ with $\psi$ below each $\phi_i$; we must show that $\psi<\psi_j$ for some $j$.

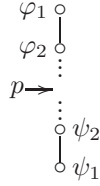
$$
\vcenter{%
\xymatrix@C=10pt@R=8pt{%
&*+={\circ}\ar@{-}[d]\ar@{}+<-10pt,0pt>*{\phi_1}\\
&*+={\circ}\ar@{}+<-10pt,0pt>*{\phi_2}\ar@{}+<0pt,-4pt>*{.}\ar@{}+<0pt,-7pt>*{.}\ar@{}+<0pt,-10pt>*{.}\\
*+={p}\ar[r]&\\
&*+={\circ}\ar@{}+<10pt,0pt>*{\psi_2}\ar@{}+<0pt,4pt>*{.}\ar@{}+<0pt,7pt>*{.}\ar@{}+<0pt,10pt>*{.}\\
&*+={\circ}\ar@{-}[u]\ar@{}+<10pt,0pt>*{\psi_1}
}}
$$

Since $\psi$ is a sum of pp formulas realised in indecomposable modules, we can assume that $\psi$ has an indecomposable free realisation $(N,n)$ say.  Since $\psi < \phi_i$ there is a morphism $f:(B_i,m) \rightarrow (N,n)$.  If $N$ were a $b'$-band module where $b'\neq b,b^{-1}$ then, since $f(m)\neq 0$, $b'$ would contain a socle pair of the form $\be^{-1}\al$, contradicting \ref{domfact}(2).  If $N$ were a $b$- or $b^{-1}$-band module then (some component of) $f$ would be induced by a matching between a factor substring and image substring of $^\infty b^\infty$ (and/or its inverse).  But, again since the image of $m$ is non-zero, that is impossible.

So $N$ must be a string module and then we can use the description of morphisms from band modules to string modules (see \cite[\S 6.3.2]{Har} for a clear description):  any morphism from $B_i$ to a string module is a linear combination of compositions of irreducible maps and maps which factor through $\Sigma(b,S)$, hence through the product module $\ov M(^\infty b^\infty)$ seen above.  Any composition of graph maps which annihilates $m$ can be ignored, and the rest are embeddings.  So any component of $f$ is essentially the composition of a map from some $B_{i+t}$ to $\ov M$ followed by a graph map induced by a factor substring $u$ of $^\infty b^\infty$.  Taking a copy of $c^k$ containing $u$ as a substring, we see that the latter map factors through the canonical map to some $M_k$.  Therefore, choosing $j$ large enough, $f$ itself will factor through $M_j$ and so $\psi_j \geq \psi$, as claimed.
\end{proof}

We also describe this basis as sets of points of the spectrum.

\begin{cor}\label{prufernbhd2}\marginpar{prufernbhd2}  Define  $U_{i,j} = \{M(b, S[m]) \: | \: m \geq i \} \cup \{\Sigma(b,S) \} \cup \mathcal{Z}_j$ where $\mathcal{Z}_j$ is the set of all string modules $C(y)$ such that $y$ contains as an image substring a finite factor string of the form $x'c^j x''$ of $^\infty b^\infty$, where the string $c$ is as in the proof of \ref{prufernbhd} (that is, all ``long enough" finite factor strings of $^\infty b^\infty$.  Then the $U_{i,i}$ form a basis of open neighbourhoods of $\Sigma(b,S)$.
\end{cor}
\begin{proof} It is not difficult to see that $U_{i,j}$ is open (cf. the last part of the proof of \ref{gencnbhd}).  It suffices to show that $U_{i,j} \subseteq (\phi_i/\psi_{j-1})$; clearly both subsets contain no $b'$-band where $b'\neq b, b^{-1}$ and they do coincide on the set of $b$-band modules, so it suffices to check the inclusion just of their intersections with string modules.  Let $N =C(y) \in U_{i,j}$ be a string module.  There is an obvious morphism from $B_i$ to $N$ sending the socle element $m$ to a linear combination $n$ of at least $j$ different `adjacent' canonical basis elements in $\al N \cap \be N$.  Then $n$ will satisfy $\phi_i$ but, since clearly this map does not factor through $M_{j-1}$, $n$ does not satisfy $\psi_{j-1}$.
\end{proof}

Suppose that $b$ is a band and $bc$ is a string; we will say that $bc$ {\bf ascends from} $b$ if $b^\infty<bc$ and $bc$ {\bf descends from} $b$ if $bc<b^\infty$ (the remaining case is that $bc$ is a substring of $b^\infty$).  For instance, the string $\be \al^{-1}\be \al$ over $X_3$ will ascend from the band $b=\be\al^{-1}$.

\begin{theorem}\label{bandsprufer}\marginpar{bandsprufer}  Let $b$ be a band.  Then the rank of each $b$-Pr\"{u}fer module is $s+t+1$ where $s$ is the indent of $b$ {\it via} strings that ascend from $b$ and $t$ is the indent of $b$ {\it via} strings that ascend from $b^{-1}$.
\end{theorem}
\begin{proof}  Because the intervals $[\phi_{i+1}, \phi_i]$ are simple, by the definition of m-dimension it suffices to prove that each interval $[\psi_i, \psi_{i+1}]$ has m-dimension $s+t$.  By \ref{m-conn}(1) this will be the maximum value of CB-rank of points in the corresponding open set $(\psi_{j+1}/\psi_j)$.  We show that this interval contains no band modules.

Any module in $(\psi_{j+1}/\psi_j)$ contains an element in the intersection ${\rm im}(\al) \cap {\rm im}(\be)$, so any band module would be a $b$- or $b^{-1}$-band module. But then there would be a map from, say, $B_1$ to $M_{j+1}$ to a $b$- or $b^{-1}$-band module, with the image of $m\in B_1$ being non-zero.  Hence there would be a factor substring of $^\infty b^\infty$ which includes the string $\be^{-1}\al$ or $\al\be^{-1}$ equal to an image substring of $^\infty b^\infty$; this (consider the ends of this substring) is impossible.

The string modules in  $(\psi_{j+1}/\psi_j)$ are obtained by taking factor substrings $d$ of $c^{j+1}$ (which are not also factor substrings of $c^j$, so which are obtained by factoring out only ``small" submodules at either side) and then embedding these as image substrings of infinite strings $w$ which, therefore, ascend from $b$.  By the computation of ranks in Section \ref{secrkstring}, the points of maximal rank in this open set will have the form $^\infty b_1 d' b_2^\infty$ where $b_2$ is the nearest band to $b$ on a path witnessing that the maximal indent of $b$ {\it via} ascending paths and $b_1$ is similarly a nearest neighbouring band on a maximal band-length path that ascends from $b^{-1}$ (and $d'$ contains some $d$ as above as an image substring).  By \ref{2stringrank} this module has rank $(s-1)+(t-1)+2= s+t$.  In the case where there is an ascending path just to one side, say from $b$, then we use \ref{1stringrank} to obtain the value $(s-1)+1=s$ for the m-dimension of the interval and hence the value $s$ for the rank of $\Sigma(b,S)$.  In the case that there is no path ascending from $b$, the interval contains only finitely many finite-dimensional points, so has rank $0$.

Thus in each case, the maximal rank of points in $(\psi_{j+1}/\psi_j)$, and hence the m-dimension of $[\psi_j,\psi_{j+1}]$, is $s+t$.
\end{proof}

\begin{example}\label{lam3pruf}\marginpar{lam3pruf} Let $R$ be the path algebra of $\Lambda_3$.  Let $b_i =\be_i\al_i^{-1}$ ($i=1, 2,3)$
be representatives of the bands.  Then the bridge quiver consists of
$b_1 \xrightarrow{\ga_1^{-1}\al_2^{-1}} b_2 \xrightarrow{\ga_2^{-1}\al_3^{-1}} b_3$ and its inverse.  Because there is no ascending path
from $b_1$,
the rank of any $b_1$-Pr\"ufer module is 1.  There is an ascending path to the left of $b_2$ (that is, for $b_2^{-1}$) of length 1,
so we obtain $1+0+1=2$ for the rank of each $b_2$-Pr\"ufer point.  Finally there is an ascending path of length 2 to the left of $b_3$,
so the corresponding Pr\"ufer points have rank $2+0+1=3$.
\end{example}

\subsubsection{Adic modules}

The ranks of adic modules could be calculated by a similar, though dual, argument, using the coray of epimorphisms in the tube with quasisimple module $M(b,S)$.  It is quicker to make use of elementary duality.  We will use that the adic left $R$-module $\Pi(b,S)$ is the elementary dual of a Pr\"ufer right $R$-module attached to a certain band $b^\ast$ for right $R$-modules.  For finite-dimensional modules (and some others) elementary duality is simply ${\rm Hom}_K(-,K)$-duality.  Given a right $R$-module $L$, we will write $L^\ast = {\rm Hom}_K(L,K)$ for its Hom-dual left module, and the same for left-to-right Hom-duality.

First observe that right modules over $R$, which has been presented as a string algebra by a quiver $Q$ with relations, are representations of the quiver which is obtained from $Q$ by reversing all arrows and, correspondingly, all relations; that is, they are left modules over the opposite algebra to $R$ and so everything that we have proved applies in this way to right $R$-modules.  Note also that the ``opposite" of a band $b$ for left $R$-modules is a cyclic permutation of a 'dual' band, $b^\ast$, for right $R$-modules.  Furthermore, the parametrisation of quasisimple $b$-band modules is, depending on our choice of $b^\ast$ and of how to parametrise $b^\ast$-band modules, either the same as or inverse to that for $b$-band modules, meaning that $M(b,S)^\ast = M(b^\ast,S)$, or $= M(b^\ast,S^{-1})$ where by $S^{-1}$ we mean the image of the simple $K[T,T^{-1}]$-module where $T$ acts as the inverse of the action on the simple module $S$.  We will set $M(b,S)^\ast = M(b^\ast, S^\ast)$ as a notation to cover either case.  In the next result we compute the elementary dual pairs of indecomposable pure-injective right and left modules.

\begin{prop} For a domestic string algebra $R$, elementary duality gives a homeomorphism between the right and left Ziegler spectra of $R$, with pairing between infinite-dimensional points being as follows (for finite-dimensional points, the pairing is just ${\rm Hom}_K(-,K)$-duality).

\noindent (1) The $b$-Pr\"ufer modules are dual to the $b^\ast$-adic modules and the $b$-adic modules are dual to the $b^\ast$-Pr\"ufer modules.

\noindent (2)  If $w$ is a contracting 1-sided or 2-sided string then $C(w)$ is dual to the expanding string module $C(w^\ast)$ and if $w$ is an expanding 1-sided or 2-sided string then $C(w)$ is dual to the contracting string module $C(w^\ast)$.

\noindent (3)  If $w$ is a mixed 2-sided string then $C(w)$ is dual to the mixed string module $C(w^\ast)$.

\noindent (4)  The $b$-generic module is elementary dual to the $b^\ast$-generic module.

\end{prop}
\begin{proof} (1) and (4) These follow from the corresponding result for $\widetilde{A_n}$-modules.

(2) For 1-sided strings this follows from (1) (considering the tubes parametrised by $0, \infty$).  If $w$ is a contracting 2-sided string then the module $C(w)$ is the direct limit of the graph map monomorphisms between the 1-sided string modules $C(w_n)$ where $w_n$ is any increasing sequence of image substrings of $w$, so is in the Ziegler-closure of these points, indeed is easily seen to be an accummulation point of minimal rank.  Dualising, we have that the elementary dual is the accummulation point of minimal rank in the Ziegler-closure of the $C(w_n)^\ast = C(w_n^\ast)$ and it is easily checked that $C(w^\ast)$ has this property and therefore is the dual of $C(w)$.

(3)  Consider a 2-sided mixed string $w$.  The corresponding indecomposable pure-injective module $C(w)$ is the direct limit of the obvious system of embeddings between 1-sided product modules $C(w_n)$ where the $w_n$ form an increasing sequence of expanding 1-sided image substrings.  Since we have already computed the elementary duals of the $C(w_n)$, we can proceed as in (2) and deduce that $C(w_n^\ast)$ is the elementary dual of $C(w)$.
\end{proof}

Using this and \ref{prufernbhd2} we obtain neighbourhood bases for adic modules.

\begin{cor}\label{adicnbhd}\marginpar{adicnbhd}  The adic module $\Pi(b,S)$ has, for a neighbourhood basis, the open sets $V_{ij} = \{M(b, S[m]) \: | \: m \geq i \} \cup \{\Pi(b,S) \} \cup \mathcal{W}_j$ where $\mathcal{W}_j$ is the set of all string modules $C(y)$ such that $y$ has, as a factor substring, a finite string of the form $x'b^j x''$ of $^\infty b^\infty$, which is an image substring of $^\infty b^\infty$.  (Of course we may limit to $j=i$ for aesthetic reasons.)
\end{cor}

It remains only to observe that, because the arrows turn around in going from left to right modules, the description of the rank of $\Sigma(b^\ast,S^\ast)$ in terms of arrows ascending from the band $b^\ast$ dualises to the same description of the rank of $\Pi(b,S)$ but in terms of arrows which descend from the band $b$.  We therefore obtain the following result for calculating the rank of adic modules using the above and that the CB-ranks elementary-dual pairs of modules are equal.

\begin{theorem}\label{bandsadic}\marginpar{bandsadic}  Let $b$ be a band.  Then the rank of each $b$-adic module is $s+t+1$ where $s$ is the indent of $b$ {\it via} strings that descend from $b$ and $t$ is the indent of $b$ {\it via} strings that descend from $b^{-1}$.
\end{theorem}

Continuing example \ref{lam3pruf}, there is a path of length 2 descending from $b_1$, therefore the rank of every $b_1$-adic module is $0+2+1=3$; similarly the rank of each $b_2$-adic module is 2 and the rank of each $b_1$-adic module is 1.

\subsubsection{Generic modules}

Suppose that $b$ is a band; first we describe a neighbourhood basis of the $b$-generic.

\begin{cor}\label{gencnbhd}\marginpar{gencnbhd}  The sets $U$ which satisfy the following conditions form a basis of open neighbourhoods of the $b$-generic $G_b$:  $U$ contains all the infinite-dimensional $b$-band modules, all but finitely many finite-dimensional $b$-band modules and no $b'$-band modules for $b' \neq b, b^{-1}$.  Furthermore there is an $n$ such that the string modules in $U$ are exactly the $C(y)$ where $y$ contains $b^n$ as an image subword.
\end{cor}
\begin{proof}
Choose an open set containing $G_b$; by what we have already mentioned we can reduce it to an open set $U$ containing $G_b$ and no $b'$-band module over any other band $b'$.  The intersection of $U$ with the image of a representation embedding giving the $b$-band modules is the homeomorphic image of an open neighbourhood of the generic in the initial ring.  It is known from \cite{PreTame} and \cite{RinTame} that any such set must contain all the infinite-dimensional points and all but finitely many finite-dimensional points.  So, it remains to determine the intersection of $U$ with the set of string modules.

We claim that there is an $n$ such that $U$ contains every string module $C(y)$ where $y$ contains $b^n$ as an image substring.  Suppose, for a contradiction, that for each $n$ there is a string $x_n$ which contains $b^n$, as an image substring but such that $C(x_n)$ is not in $U$.  We must show that $G_b$ is in the closure of the $C(x_n)$.  It is easy to see that there is some band $b_0$ and finite string $c$ such that, for each $n$, $b_0^ncb^n$ occurs as an image substring of some $x_{m(n)}$.  Therefore, by \ref{stringnbhd}, the module $C(^\infty b_0cb^\infty)$ is in the closure of this set. First, one can check that the closure of $C(^\infty b_0cb^\infty)$ consists only of that module plus generic points:  for we have computed neighbourhood bases of all but the generic points and, by inspection, none of them but $G_{b_0}$ and $G_b$ is in the closure of this point.

Choose a socle pair $\alpha^{-1} \beta$ for $b$. Let $\phi$ be the pp formula whose solution set is ${\rm im}(\alpha) \, \cap\, {\rm im}(\beta)$.  This solution set in $C(^\infty b_0cb^\infty)$ is, note, infinite-dimensional.  Indeed, since there is a contracting or expanding endomorphism of $C(^\infty b_0cb^\infty)$ (see Section \ref{secstring}) acting on the $b^\infty$ part of $C(^\infty b_0cb^\infty)$, that solution set is of infinite length over the endomorphism ring of $C(^\infty b_0cb^\infty)$ and so must have an infinite (ascending or descending) chain of pp-definable subgroups.  This implies (by compactness in, say, the model-theoretic sense\footnote{This can be said in terms of the pp-sort $\phi$ evaluated at $M$ with its induced model-theoretic structure:  since it has infinite endolength there is a non-isolated point in its Ziegler-closure.  Alternatively, and more directly, one can produce an irreducible pp-type which is not realised in $M$.}) that one of the points in the Ziegler-closure of $M_0$ also has non-zero solution set for $\phi$.  Those points all are generic, and the only generic point which satisfies this is $G_b$, as required.

Thus we have shown that, for any open set $U$ containing $G_b$, there is $n$ such that every string module containing $b^n$ as an image substring is in $U$.  It remains to note that the set of string modules containing $b^n$ as an image substring is open - the proof of \ref{stringnbhd} applied with $b^n$ in place of $w_n$ there, shows exactly this.  Therefore this set, together with the intersection of $U$ with the $b$-band modules is an open set (being the union of an open set with an open subset of the complement of that set) and contains $G_b$.  Furthermore, we have just shown that every open set containing $G_b$ contains such a set, so this does give a basis.
\end{proof}

Finally we compute the CB-rank of each generic module.  We will see that the rank of a generic module is determined not just by the ranks of modules supported on the same band, but also by the 2-sided strings which pass through $b$, that is which contain $b$ as a substring.  Consider the band $b_2= \be_2\al_2^{-1}$ over $\Lambda_3$ from Example \ref{lam3pruf}.  The ranks of the $b_2$-Pr\"{u}fers and $b_2$-adics are 2 but, as we will see below, the rank of the $b_2$-generic is 4.  This is because there is a 2-sided string passing through $b_2$ such that the corresponding mixed module has rank 3.

\begin{theorem}\label{genc}\marginpar{genc}  If $b$ is a band then the rank of the associated generic $G_b$ is $s+t+2$ where $s$, respectively $t$, is the left, resp. right, indent of $b$.
\end{theorem}
\begin{proof}  Consider a band $b_0$ adjacent to $b$ on a longest path containing $b$ in the bridge quiver.  Let $w_0$ be a 2-sided infinite string of the form $^\infty b_0c_0b^\infty$.  The corresponding string module $M_0=C(w_0)$ has rank $s-1+t+2 = s+t+1$.  We have shown already in the proof of \ref{gencnbhd} that $G_b$ is in the closure of this module, hence has rank at least $s+t+2$.

But now observe that every point in an open neighbourhood $U$, as in \ref{gencnbhd}, of $G_b$, apart from $G_b$ itself, has been assigned a CB-rank but the results already proved, and $s+t+1$ is the maximal value attained,  Hence the CB-rank of $G_b$ is exactly $s+t+2$, as claimed.
\end{proof}

\begin{example}  As an example consider the following 3-domestic string algebra.

$$
X_4 \hspace{1cm}
\vcenter{%
\def\labelstyle{\displaystyle}
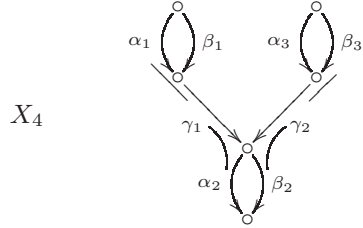
\xymatrix@C=20pt@R=20pt{%
*+={\circ}\ar@/_0.5pc/[d]_{_{\al_1}}="al1"\ar@/^0.5pc/[d]^{_{\be_1}}&&
*+={\circ}\ar@/^0.5pc/[d]^{_{\be_3}}="be3"\ar@/_0.5pc/[d]_{_{\al_3}}\\
*+={\circ}\ar[rd]_{_{\ga_1}}="ga1"&&*+={\circ}\ar[ld]^{_{\ga_2}}="ga2"\\
&*+={\circ}\ar@/_0.5pc/[d]_{_{\al_2}}="al2"\ar@/^0.5pc/[d]^{_{\be_2}}="be2"&\\
&*+={\circ}&
\ar@{-}"al1"+<4pt,-8pt>;"ga1"+<-2pt,10pt>
\ar@{-}"be3"+<-6pt,-12pt>;"ga2"+<3pt,10pt>
\ar@/^0.2pc/@{-}"ga1"+<6pt,0pt>;"al2"+<6pt,6pt>
\ar@/_0.2pc/@{-}"ga2"+<-6pt,0pt>;"be2"+<-6pt,6pt>
}}
$$\label{X4}

\vspace{3mm}

\noindent with short relations $\ga_1\al_1= \al_2\ga_1= \be_2\ga_2= \ga_2\be_3= 0$.

Choose $b_i= \al_i\be_i^{-1}$, $i= 1, 2, 3$ and their inverses as the vertices of the bridge quiver. The bridge quiver of $X_4$ consists of the path

$$
b_1\xr{\ga_1^{-1}\be_2^{-1}} b_2\xr{\al_2\ga_2} b_3 $$ and its inverse.

Since the left indent of $b_1$ is 0 and the right indent is $2$, the rank of $G_{b_1}$ is $0+2+2 =4$.  Similarly the ranks of $G_{b_2}$ and $G_{b_3}$ are 4.
\end{example}

\begin{example}
To get an example where the generics do not all have the same rank, consider the quiver $X_5$.

$$
X_5 \hspace{1cm}
\vcenter
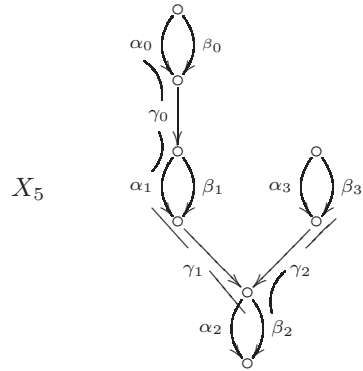
{%
\xymatrix@C=20pt@R=20pt{%
*+={\circ}\ar@/_0.5pc/[d]_{\al_0}="al"\ar@/^0.5pc/[d]^{_{\be_0}}&&\\
*+={\circ}\ar[d]_{\ga_0}="ga"&&\\
*+={\circ}\ar@/_0.5pc/[d]_{\al_1}="al1"\ar@/^0.5pc/[d]^{\be_1}&&
*+={\circ}\ar@/^0.5pc/[d]^{\be_3}="be3"\ar@/_0.5pc/[d]_{\al_3}\\
*+={\circ}\ar[rd]_{\ga_1}="ga1"&&*+={\circ}\ar[ld]^{\ga_2}="ga2"\\
&*+={\circ}\ar@/_0.5pc/[d]_{\al_2}="al2"\ar@/^0.5pc/[d]^{\be_2}="be2"&\\
&*+={\circ}& \ar@{-}"al1"+<4pt,-8pt>;"ga1"+<-2pt,9pt>
\ar@{-}"be3"+<-6pt,-12pt>;"ga2"+<2pt,9pt>
\ar@{-}"ga1"+<6pt,0pt>;"be2"+<-14pt,6pt>
\ar@/_0.2pc/@{-}"ga2"+<-6pt,0pt>;"be2"+<-4pt,6pt>
\ar@/^0.2pc/@{-}"al"+<1pt,-6pt>;"ga"+<1pt,6pt>
\ar@/^0.2pc/@{-}"ga"+<0pt,-6pt>;"al1"+<4pt,6pt> }}
$$\label{X5}

\vspace{3mm}

Choose $\al_i\be_i^{-1}$ and their inverses for $i=0,\dots ,3$ as the vertices of the bridge quiver. Then the bridge quiver is the following directed
graph and its inverse.

$$
\vcenter{%
\xymatrix@C=20pt@R=20pt{%
*+{b_2}\ar[r]^{\al_2\ga_1}\ar[dr]_{\al_2\ga_2}&*+{b_1^{-1}}\ar[r]^{\be_1\ga_0}&*+{b_0^{-1}}\\
&*+{b_3}&
}}
$$

From the bridge quiver we see that the rank of $G_{b_2}$ is $0+2+2=4$ and $G_{b_0}, G_{b_1}$ also have rank 4.  On the other hand $G_{b_3}$ has rank $1+0+2=3$.
\end{example}

In particular we have the following theorem confirming a conjecture of Schr\"oer.

\begin{theorem}\label{KG-dim}\marginpar{KG-dim}
Let $R$ be a domestic string algebra.  The Krull--Gabriel dimension of the category of $R$-modules is $n+2$ where $n$ is the maximal length of a path in the bridge quiver of $R$.
\end{theorem}

The following is immediate.

\begin{cor} If $R$ is a domestic string algebra then its Ziegler spectrum is a $T_0$ space.
\end{cor}

\end{document}